\documentclass[11pt]{article}
\usepackage{latexsym}
\usepackage{amsfonts,nccbbb,xfrac,mathtools}
\usepackage{amsmath}
\usepackage{amssymb}
\usepackage{wasysym}
\usepackage{bbm}
\usepackage{color}
\usepackage{graphicx}
\usepackage{lipsum}
\usepackage{environ}
\usepackage{epsfig}
\usepackage{hyperref}
\usepackage[tmargin=1.25in,bmargin=1.25in,lmargin=1.1in,rmargin=1.1in]{geometry}
\usepackage{amsthm}
\newtheorem{thm}{Theorem}[section]
\newtheorem{lem}[thm]{Lemma}

\makeatletter\@addtoreset{equation}{section}\makeatother

\usepackage{tikz}
\usepackage{subcaption}
\usepackage{pgfplots}
\pgfplotsset{compat=newest}
\usetikzlibrary{plotmarks}
\usetikzlibrary{arrows.meta}
\usepgfplotslibrary{patchplots}
\usepackage{grffile}

\newcommand{\musc}{\mu_{\text{sc}}}
\newcommand{\muas}{\mu_{\text{arc}}}

\newcommand{\EE}{\mathbb{E}}

\newcommand{\RR}{\mathbb{R}}

\newcommand{\NN}{\mathbb{N}}

\begin{document}

\title{Nonstandard Large and Moderate Deviations for the Laguerre Ensemble}

\author{Helene Götz, Jan Nagel}

\maketitle

\begin{abstract}
	In this paper, we show limit theorems for the weighted spectral measure of the Laguerre ensemble under a nonstandard scaling, when the parameter grows faster than the matrix size. For this parameter scaling, the limit behavior is similar to the case of the Gaussian ensemble. We show a large deviation principle, moderate deviations and a CLT for the spectral measure. For the moderate deviations and the CLT, we observe a particular dependence on the rate of the parameter and a corrective shift by a signed measure. The proofs are based on the tridiagonal representation of the Laguerre ensemble. 
\end{abstract}





\section{Introduction}
The Laguerre ensemble, or Wishart ensemble, is one of the classical distributions in random matrix theory. The eigenvalues $\lambda_1,\dots,\lambda_n$ of the Laguerre ensemble of size $n$ have a joint Lebesgue density given by  
\begin{align}
	\label{DensityLaguerre}
	f_{\beta,\gamma}(\lambda_1,...,\lambda_n) = Z_{\beta,\gamma}^L \prod_{1 \leq i<j \leq n} |\lambda_i-\lambda_j|^{\beta} \prod_{i=1}^n\lambda_i^{\gamma-(n-1)\frac{\beta}{2}-1} e^{-\sum_{i=1}^n \frac{\lambda_i}{2}}\mathbbm{1}_{\{\lambda_i>0\}},
\end{align}
where $\beta>0$ and $\gamma>(n-1)\frac{\beta}{2}$. This distribution was introduced by Wishart \cite{wishart1928} in the study of sample covariance matrices for Gaussian data. When $Z_n$ is a $n\times p$ matrix with independent standard normal entries, $p\geq n$, the eigenvalues of the random matrix
\begin{align}
	\label{eq:wishartmatrix}
	X_n = Z_n Z_n^* 
\end{align}
have the density \eqref{DensityLaguerre} with $\beta=1$ and $\gamma=p/2$, where $n$ corresponds to the sample size and $p$ to the dimension of the data. Since then, it has found numerous applications in multivariate statistics, see \cite{muirhead2009aspects,bai2009random,dieng2011application,johnstone2001distribution} or random matrix theory in general \cite{mehta2004random,anderson2010introduction}. For $\beta=2$ or $\beta=4$, it is also possible to give a full complex or quaternion matrix model with eigenvalue density  \eqref{DensityLaguerre}. For general $\beta>0$, the $\beta$-ensemble with this density is a model of a log-gas at temperature $\beta^{-1}$ \cite{forrester2010log,saff2013logarithmic}.

In this paper, we consider a nonstandard scaling of the Laguerre ensemble, when the parameter $\gamma=\gamma_n$ grows at a rate faster than $n$. In the statistical context, this corresponds to the situation of very high-dimensional data. While most limit theorems deal with the case of $\gamma_n$ growing at most linearly, such nonstandard scaling reveals interesting relations between the Laguerre ensemble and the other two classical distributions of invariant random matrices, the Gaussian and the Jacobi ensemble. 
The classical result of \cite{bai1988convergence} shows that the empirical distribution of the eigenvalues converges to the semicircle law almost surely  and not to the usual equilibrium measure of the Laguerre ensemble, the Marchenko-Pastur law. \cite{jiang2015approximation} showed that the variation distance between the eigenvalue distribution of the standardized Laguerre and Gaussian ensemble vanishes if $\gamma_n\gg n^3$. This approximation was further refined by \cite{racz2019smooth,chetelat2019middle}. Moreover, \cite{jiang2015approximation} show a large deviation principle (LDP) for the empirical measure when $\gamma_n\gg n^2$, where the rate function is as in the Gaussian case. A similar phenomenon was observed  for the Jacobi ensemble under a nonstandard scaling \cite{dette2009some,jiang2013limit,nagel2014nonstandard} where, depending on the parameters, the eigenvalues behave similar to the Laguerre or the Gaussian case. Closely related are asymptotics for the zeros of Jacobi polynomials, which can exhibit asymptotics similar to those of Laguerre or Hermite polynomials \cite{dette1995some}. 

We consider asymptotics for a rescaled version of the weighted (or projected) spectral measure $\mu_n$ of a matrix $X_n$ of the Laguerre ensemble. The spectral measure can be  characterized by the fact that its $k$-th moment is given by $(X_n^k)_{1,1}$. It is given as
\begin{align*}
	\mu_n = \sum_{i=1}^n w_i \delta_{\lambda_i} ,  
\end{align*}
supported by the eigenvalues $\lambda_1,\dots ,\lambda_n$ and with weights related to the eigenvectors of $X_n$ (see \eqref{eq:spectralmeasure}). Under the standard scaling, large deviation principles for the sequence $(\mu_n)_n$ have been proven by \cite{gamboa2011large,magicrules}. The large deviation behavior is very different from the one of the empirical eigenvalue measure considered by \cite{arous1997large}, with both a different speed and rate function. Moderate deviations for empirical eigenvalue measures have been proven for Wigner-type random matrices in \cite{dembo2003moderate,doering2013moderate} and for the weighted spectral measures of classical ensembles in \cite{nagel2013moderatedeviationsspectralmeasures}. 

Under the nonstandard scaling $\gamma_n\gg n$, we show a large deviation principle for the sequence of spectral measures $\mu_n$ (Theorem \ref{thm:LDPLaguerre}) with a rate function as in the Gaussian case. This also shows the exponentially fast convergence to the semicircle law $\mu_{sc}$. Furthermore, we show a moderate deviation principle (Theorem \ref{thm:MainMDP}) for the moments when we center $\mu_n$ by $\mu_{sc}$. Interestingly, the rate function in the moderate deviation principle depends on the specific growth of $\gamma_n$, and with a minimum not necessarily at the zero measure. Instead, we observe the convergence to the moments of a signed measure, which has a polynomial density with respect to the arcsine distribution. Additionally, we obtain a central limit theorem (Theorem \ref{thm:CltPolynoms}) for polynomial test functions,  with the limit distribution similarly sensitive to the growth of $\gamma_n$. These results extend the LDP in \cite{jiang2015approximation} to the weighted spectral measure, under a weaker condition on $\gamma_n$. To the best of our knowledge, a moderate deviation principle has not been proven previously under nonstandard scaling.

The proofs rely on the tridiagonal representation of the Laguerre ensemble by \cite{Dumitriu_2002}. This enables us to start with large deviation principles for independent random variables, which can then be transferred to the spectral measure, as in the Jacobi case in \cite{nagel2014nonstandard}. This method has the advantage that we do not need extra conditions on the sequence $\gamma_n$, in particular no additional lower bound for its growth. 

The following Section \ref{sec:results} contains our main results.
In Section \ref{sec:prelim}, we introduce the tridiagonal model and its spectral measure. The proofs can be found in Section \ref{sec:proofs}.

\section{Results}
\label{sec:results}

For our limit theorems, we rescale the eigenvalues of the Laguerre ensemble with parameter $\gamma_n$, dimension $n$ and inverse temperature $\beta>0$ and set 
\begin{align}
	\label{eq:ScalingLDP}
	\tilde \lambda_i = \frac{1}{\sqrt{2\gamma_nn\beta}} (\lambda_i-2\gamma_n)
\end{align}
for $i=1,\dots ,n$, where $(\lambda_1,\dots ,\lambda_n)$ have the density \eqref{DensityLaguerre}. This is the same rescaling of eigenvalues as in \cite{jiang2015approximation}.
Our main random object is then the weighted spectral measure
\begin{align} \label{eq:spectralmeasure}
	\mu_n = \sum_{i=1}^n w_i \delta_{\tilde \lambda_i}
\end{align}
of the rescaled Laguerre ensemble, where the weights $(w_1,\dots ,w_n)$ are independent of the eigenvalues and follow a Dirichlet distribution on the standard simplex with homogeneous parameter $\beta/2$, that is, $(w_1,\dots ,w_{n-1})$ have the Lebesgue density
\begin{align} \label{eq:dirdensity}
	\frac{\Gamma(n\beta/2 )}{\Gamma(\beta/2)^n} \big[ w_1\dots w_{n-1}(1-w_1-\dots -w_{n-1})\big]^{\beta/2-1} 
\end{align}
on $(0,1)^{n-1}$. 
For the relation of $\mu_n$ to the random matrix model, see Section \ref{sec:RandTriMod}. The empirical eigenvalue measure of the rescaled Laguerre ensemble will be denoted by
\begin{align} \label{eq:empmeasure}
	\bar\mu_n = \frac{1}{n}\sum_{i=1}^n \delta_{\tilde \lambda_i} .
\end{align}
In the following, we will always assume that 
\begin{align} 
	\lim_{n\to \infty} \frac{\gamma_n}{n} = \infty . 
\end{align}

\subsection{Nonstandard Large Deviations for the Laguerre Ensemble} 
Let us start by recalling the definition of a large deviation principle. Let
$\mathbb{X}$ be a topological Hausdorff space with Borel $\sigma$-algebra, $\mathcal{I}: \mathbb{X} \to [0,\infty]$ a lower semicontinuous function and $(a_n)_n$ a sequence of positive real numbers with $a_n \to \infty$.
A sequence of $\mathbb{X}$-valued random variables $(X_n)_n$ satisfies a \textit{large deviation principle} (LDP) with speed $a_n$ and rate function $\mathcal{I}$ if
\begin{align*}
	\limsup_{n \to \infty} \frac{1}{a_n} \log P(X_n \in C) \leq -\inf_{x \in C} \mathcal{I}(x) \quad \text{for all closed sets } C \subset \mathbb{X}
\end{align*}
and 
\begin{align*}
	\liminf_{n \to \infty} \frac{1}{a_n} \log P(X_n \in O) \geq -\inf_{x \in O} \mathcal{I}(x) \quad \text{for all open sets } O \subset \mathbb{X}.
\end{align*}
A rate function is called \textit{good} if its level sets $\{ x\in \mathbb{X} \mid \mathcal{I}(x) \leq c \}$ are compact for all $c \geq 0$.

We consider $\mu_n$ as a random element of the space $\mathcal{M}_1^c$ of probability measure on $\mathbb R$ with compact support, equipped with the weak topology and the corresponding $\sigma$-algebra. 

In order to formulate the rate function in our LDP, we need to introduce some notation. We denote by $\mu_{sc}$ the semicircle law with density given by 
\begin{align}
	\frac{1}{2\pi}\sqrt{4-x^2} \bbbone_{[-2,2]}(x) ,
\end{align}
and by $\mathcal K(\nu|\mu)$ the Kullback-Leibler distance, or relative entropy, of $\nu$ with respect to $\mu$. For $ |x|\geq 2$, define
\begin{align}
	\mathcal F(x) = \int_2^{|x|} \sqrt{y^2-4}\, dy . 
\end{align}
Let $\mathcal S_1([-2,2])$ be the subset of probability measures $\mu$ with support given by $[-2,2]\cup E(\mu)$, where $E(\mu)$ is an at most countable subset of $\mathbb{R}\setminus [-2,2]$. 
We then have the following LDP. 

\begin{thm}
	\label{thm:LDPLaguerre}
	The sequence of spectral measures $\mu_n$ as in \eqref{eq:spectralmeasure} satisfies a large deviation principle with speed $ n\beta/2$ and good rate function $\mathcal I$. For a measure $\mu\in \mathcal S_1([-2,2])$, the rate is given by
	\begin{align} \label{eq:sumrule} 
		\mathcal{I}(\mu) = \mathcal K(\mu_{sc}|\mu) + \sum_{\lambda\in E(\mu)} \mathcal F(\lambda) . 
	\end{align}
	If $\mu \notin \mathcal S_1([-2,2])$, we have $\mathcal I(\mu)=+\infty$. 
\end{thm}

Let us remark that $\mathcal I(\mu)$ vanishes if and only if $\mu$ is equal to the semicircle law $\mu_{sc}$. The upper bound of the LDP in Theorem \ref{thm:LDPLaguerre} therefore implies the almost sure convergence of $\mu_n$ to $\mu_{sc}$ by the Borel-Cantelli lemma.

The rate function $\mathcal I$ in Theorem \ref{thm:LDPLaguerre} coincides with the rate function for the spectral measure of the Gaussian ensemble as obtained in \cite{gamboa2011large}. The specific form of the rate is identical to one side in the famous sum rule of Killip and Simon \cite{killip2003sum}, which allows to rewrite \eqref{eq:sumrule} as a function of the recursion coefficients, or Jacobi coefficients, of $\mu$. A large deviation principle for the spectral measure of the Gaussian ensemble with rate function \eqref{eq:sumrule} was proven in \cite{magicrules}, without relying on the Killip-Simon sum rule.  

The LDP obtained by \cite{jiang2015approximation} for the empirical measure $\bar \mu_n$ under the same scaling and under the condition $\gamma_n \gg n^2$ is quite different: the speed therein is equal to $n^2$ and the rate function agrees with the one of \cite{arous1997large} for the Gaussian ensemble.

\subsection{Nonstandard Moderate Deviations for the Laguerre Ensemble}

We now consider the difference 
\begin{align}\label{eq:signeddifference}
	\nu_n = \sqrt{n \beta/2b_n}(\mu_n-\musc)
\end{align}
when the spectral measure $\mu_n$ as in \eqref{eq:spectralmeasure} is centered by its weak limit, the semicircle law $\mu_{sc}$ and $1\ll b_n \ll n$. A large deviation principle for random variables centered as in \eqref{eq:signeddifference} and with a scaling between that of a LLN and a CLT is also called a \textit{moderate deviation principle} (MDP). 
Our moderate deviation principle will not hold for the measures themselves, but for their moments.

Let us denote by $\mathcal M_0^f$ the set of all signed Borel measures on $\mathbb R$ with finite variation, a total mass of zero and finite moments of all order. We write
\begin{align*}
	m_k(\nu) = \int x^k \, d\nu(x)
\end{align*}
for the $k$-th moment of $\nu\in \mathcal M_0^f$ and $m(\nu)=(m_k(\nu))_{k\geq 1}$ for its moment sequence. Then $m(\nu)$ is an element of the sequence space $\mathbb R^{\mathbb N}$, which we consider with the product topology. As shown by \cite{boas1939stieltjes}, the moment problem for signed measures has a surprisingly trivial solution: any sequence $m\in \mathbb R^{\mathbb N}$ is the moment sequence of some, and then of infinitely many $\nu \in \mathcal M_0^f$. There is however at most one representing measure with compact support, which follows from the consideration of \cite{hausdorff1923momentprobleme} for the Hausdorff moment problem for signed measures.  

A particular role in the MDP is played by the signed measure $\rho_\xi\in \mathcal M_0^f$ for $\xi\geq 0$, which is defined by its (signed) Lebesgue density
\begin{align}
	\label{eq:densityStrangemoments}
	\frac{\xi}{2\pi} \frac{x(x^2-3)}{\sqrt{4-x^2}} \bbbone_{[-2,2]}(x).
\end{align}
The moments of $\rho_\xi$ can be calculated with help of the moments of the arcsine distribution $\muas$ on $[-2,2]$, which are given by $m_{2k-1}(\muas)=0$ and $m_{2k}(\muas) = \binom{2k}{k}$ for all $k\geq 1$ \cite{Kemperman1982}. We have
\begin{align*}
	m_k(\rho_\xi) = \frac{\xi}{2}\big( m_{k+3}(\muas) - 3 m_{k+1}(\muas)\big) , 
\end{align*} 
and a short computation shows then 
\begin{align} \label{eq:strangemoments}
	m_k(\rho_\xi) = \begin{cases}
		\xi\binom{k}{\frac{k-3}{2}} \ & k\text{ odd and } k \geq 3, \\
		\;0 & k \text{ even or } k=1 .
	\end{cases} 
\end{align}

In order to formulate the following rate function more easily, we write $\sigma_m$ for a measure in $\mathcal M_0^f$ with moment sequence $m$, the specific choice of $\sigma_m$ will not be relevant.

\begin{thm}
	\label{thm:MainMDP}
	Suppose 
	$(b_n)_{n\geq 0}$ is a sequence with $b_n\to \infty$, $b_n/n\to 0$ and 
	\begin{align*}
		\frac{n\beta}{2\sqrt{b_n\gamma_n}} \xrightarrow[n \to \infty]{} \xi \in [0,\infty).
	\end{align*}
	For $\nu_n$ as in \eqref{eq:signeddifference}, the sequence of moment sequences $m(\nu_n)$
	satisfies a large deviation principle with speed $b_n$ and good rate function    
	\begin{align*}
		\mathcal I(m) =  \frac{1}{2}\sum_{k=0}^\infty \left( \int q_k(x) \, d(\sigma_m-\rho_\xi)(x) \right)^2 , 
	\end{align*}
	where $q_k, k\geq 0$ are the normalized polynomials orthogonal with respect to $\musc$. 
\end{thm}

The rate function in Theorem \ref{thm:MainMDP} is zero if and only if $m=m(\rho_\xi)$ which shows that, under our nonstandard scaling, the moments of $\nu_n$ do not converge to 0 but to the moments of $\rho_\xi$. 
It can be seen in simulations that under nonstandard scaling the empirical distribution of the rescaled eigenvalues exhibits a slight asymmetry similar to the density of the signed measure $\rho_\xi$, see Figure \ref{fig:Histogramm}. 
Note however, that while $\mu_n-\musc$ converges to zero, the enlarged difference $\nu_n =  \sqrt{n \beta/2b_n}(\mu_n-\musc)$ does not converge weakly: the total variation of $\nu_n$ is equal to $2\sqrt{n\beta/2b_n}\to \infty$ and weak convergence of signed measures implies bounded total variation  \cite[Proposition 1.4.4]{bogachev2018weak}. 

A similar signed measure appears in the limits of empirical spectral measures in \cite{enriquez2016spectra} and \cite{noiry2018spectral}, we also refer to \cite{heiny2022limiting} for a large class of perturbative signed measures.

\definecolor{mycolor1}{rgb}{0.38820,0.60390,0.00000}%
\definecolor{mycolor2}{rgb}{0.85000,0.32500,0.09800}%

\begin{figure}[h] 
	\centering
	\begin{subfigure}[t]{0.49\textwidth}
		\begin{tikzpicture}[scale = 0.95,baseline=(current bounding box.north)]
			\begin{axis}[xmin =-2.2, xmax = 2.2, ymin = 0, ymax = 0.42, axis lines = left, grid = both,minor grid style={gray!25}, major grid style={gray!25},legend style={at={(0.92,0.94)}},legend cell align={left}]
				\addplot[ybar interval, fill=mycolor1, fill opacity=0.6, draw=black, area legend] table[row sep=crcr] {%
					x	y\\
					-2.01	0.0086848635235732\\
					-1.9697	0.0607940446650124\\
					-1.9294	0.0856079404466501\\
					-1.8891	0.109181141439206\\
					-1.8488	0.126550868486352\\
					-1.8085	0.140198511166253\\
					-1.7682	0.155086848635236\\
					-1.7279	0.168734491315136\\
					-1.6876	0.174937965260546\\
					-1.6473	0.182382133995037\\
					-1.607	0.198511166253102\\
					-1.5667	0.203473945409429\\
					-1.5264	0.213399503722084\\
					-1.4861	0.218362282878412\\
					-1.4458	0.224565756823821\\
					-1.4055	0.230769230769231\\
					-1.3652	0.239454094292804\\
					-1.3249	0.246898263027295\\
					-1.2846	0.248138957816377\\
					-1.2443	0.254342431761787\\
					-1.204	0.260545905707196\\
					-1.1637	0.264267990074442\\
					-1.1234	0.269230769230769\\
					-1.0831	0.270471464019851\\
					-1.0428	0.280397022332506\\
					-1.0025	0.279156327543424\\
					-0.9622	0.281637717121588\\
					-0.9219	0.287841191066997\\
					-0.8816	0.289081885856079\\
					-0.8413	0.295285359801489\\
					-0.801	0.296526054590571\\
					-0.7607	0.300248138957816\\
					-0.7204	0.301488833746898\\
					-0.6801	0.300248138957816\\
					-0.6398	0.305210918114144\\
					-0.5995	0.305210918114144\\
					-0.5592	0.30893300248139\\
					-0.5189	0.311414392059553\\
					-0.4786	0.312655086848635\\
					-0.4383	0.310173697270471\\
					-0.398	0.315136476426799\\
					-0.3577	0.315136476426799\\
					-0.3174	0.316377171215881\\
					-0.2771	0.315136476426799\\
					-0.2368	0.316377171215881\\
					-0.1965	0.320099255583127\\
					-0.1562	0.317617866004963\\
					-0.1159	0.317617866004963\\
					-0.0755999999999997	0.317617866004963\\
					-0.0352999999999997	0.320099255583127\\
					0.00500000000000034	0.317617866004964\\
					0.0453000000000001	0.317617866004961\\
					0.0856000000000003	0.317617866004961\\
					0.125900000000001	0.317617866004964\\
					0.1662	0.316377171215883\\
					0.2065	0.315136476426797\\
					0.2468	0.315136476426797\\
					0.287100000000001	0.312655086848637\\
					0.3274	0.312655086848637\\
					0.3677	0.311414392059552\\
					0.408	0.307692307692306\\
					0.448300000000001	0.308933002481391\\
					0.4886	0.303970223325064\\
					0.5289	0.30397022332506\\
					0.5692	0.305210918114142\\
					0.609500000000001	0.297766749379654\\
					0.6498	0.296526054590572\\
					0.6901	0.295285359801487\\
					0.7304	0.292803970223323\\
					0.770700000000001	0.289081885856081\\
					0.811	0.285359801488835\\
					0.8513	0.285359801488832\\
					0.8916	0.279156327543423\\
					0.931900000000001	0.277915632754344\\
					0.9722	0.271712158808934\\
					1.0125	0.271712158808931\\
					1.0528	0.263027295285358\\
					1.0931	0.264267990074443\\
					1.1334	0.253101736972706\\
					1.1737	0.253101736972703\\
					1.214	0.248138957816376\\
					1.2543	0.243176178660051\\
					1.2946	0.236972704714641\\
					1.3349	0.230769230769229\\
					1.3752	0.22456575682382\\
					1.4155	0.220843672456577\\
					1.4558	0.210918114143922\\
					1.4961	0.20471464019851\\
					1.5364	0.197270471464019\\
					1.5767	0.188585607940448\\
					1.617	0.18362282878412\\
					1.6573	0.172456575682381\\
					1.6976	0.163771712158808\\
					1.7379	0.151364764267991\\
					1.7782	0.141439205955336\\
					1.8185	0.129032258064515\\
					1.8588	0.110421836228287\\
					1.8991	0.0942928039702238\\
					1.9394	0.0744416873449136\\
					1.9797	0.0359801488833745\\
					2.02	0.0359801488833745\\
				};
				\addlegendentry{$\bar\mu_n$};
				\addplot [color=mycolor2]
				table[row sep=crcr]{%
					-2	0\\
					-1.95	0.0707300038338783\\
					-1.9	0.0993922301044098\\
					-1.85	0.120947285576661\\
					-1.8	0.138748062660508\\
					-1.75	0.154101111015375\\
					-1.7	0.167680137347119\\
					-1.65	0.179887334256116\\
					-1.6	0.190985931710274\\
					-1.55	0.201159508145051\\
					-1.5	0.210542199673896\\
					-1.45	0.21923557304948\\
					-1.4	0.227318727027916\\
					-1.35	0.234854677077757\\
					-1.3	0.241894571153323\\
					-1.25	0.248480575261024\\
					-1.2	0.254647908947033\\
					-1.15	0.260426318371607\\
					-1.1	0.265841166094589\\
					-1.05	0.270914252856727\\
					-1	0.275664447710896\\
					-0.95	0.280108178357997\\
					-0.9	0.284259817692836\\
					-0.85	0.288131992057356\\
					-0.8	0.291735829578\\
					-0.75	0.295081162042978\\
					-0.7	0.298176690313229\\
					-0.65	0.301030120785472\\
					-0.6	0.303648278629284\\
					-0.55	0.306037202198337\\
					-0.5	0.30820222203075\\
					-0.45	0.310148027110353\\
					-0.4	0.31187872049347\\
					-0.35	0.313397865967992\\
					-0.3	0.314708527069708\\
					-0.25	0.315813299510844\\
					-0.2	0.31671433785971\\
					-0.15	0.317413377135164\\
					-0.0999999999999999	0.317911749835126\\
					-0.0499999999999998	0.318210398797024\\
					0	0.318309886183791\\
					0.0499999999999998	0.318210398797024\\
					0.0999999999999999	0.317911749835126\\
					0.15	0.317413377135164\\
					0.2	0.31671433785971\\
					0.25	0.315813299510844\\
					0.3	0.314708527069708\\
					0.35	0.313397865967992\\
					0.4	0.31187872049347\\
					0.45	0.310148027110353\\
					0.5	0.30820222203075\\
					0.55	0.306037202198337\\
					0.6	0.303648278629284\\
					0.65	0.301030120785472\\
					0.7	0.298176690313229\\
					0.75	0.295081162042978\\
					0.8	0.291735829578\\
					0.85	0.288131992057356\\
					0.9	0.284259817692836\\
					0.95	0.280108178357997\\
					1	0.275664447710896\\
					1.05	0.270914252856727\\
					1.1	0.265841166094589\\
					1.15	0.260426318371607\\
					1.2	0.254647908947033\\
					1.25	0.248480575261024\\
					1.3	0.241894571153323\\
					1.35	0.234854677077757\\
					1.4	0.227318727027916\\
					1.45	0.21923557304948\\
					1.5	0.210542199673896\\
					1.55	0.201159508145051\\
					1.6	0.190985931710274\\
					1.65	0.179887334256116\\
					1.7	0.167680137347119\\
					1.75	0.154101111015375\\
					1.8	0.138748062660508\\
					1.85	0.120947285576661\\
					1.9	0.0993922301044098\\
					1.95	0.0707300038338783\\
					2	0\\
				};
				\addlegendentry{$\musc$};
			\end{axis}
		\end{tikzpicture}
		
	\end{subfigure}
	\begin{subfigure}[t]{0.48\textwidth}
		\begin{tikzpicture}[scale = 0.95,baseline=(current bounding box.north)]
			\begin{axis}[xmin = -2.2, xmax = 2.2, ymin = -1.2, ymax = 1.2, axis lines = middle, grid = both,minor grid style={gray!25}, major grid style={gray!25},legend style={at={(0.92,0.94)}},legend cell align={left}]
				\addplot[ybar, bar width=0.032, fill=mycolor1, draw=black, area legend] table[row sep=crcr] {%
					-1.98985	-1.35382545085982\\
					-1.94955	-0.594433467846431\\
					-1.90925	-0.532876642431841\\
					-1.86895	-0.240472389501479\\
					-1.82865	-0.136793845972152\\
					-1.78835	-0.134228919393059\\
					-1.74805	0.0247144906264816\\
					-1.70775	0.177536161137452\\
					-1.66745	-0.0479253941980906\\
					-1.62715	-0.156725433880717\\
					-1.58685	0.276417660415316\\
					-1.54655	0.0952847960323789\\
					-1.50625	0.231441047918816\\
					-1.46595	0.106145511604432\\
					-1.42565	0.0766529012349968\\
					-1.38535	0.0689156180155042\\
					-1.34505	0.225089634580249\\
					-1.30475	0.327817031552856\\
					-1.26445	0.087997014970738\\
					-1.22415	0.152128310715852\\
					-1.18385	0.231426054926707\\
					-1.14355	0.181135881521175\\
					-1.10325	0.216382601831112\\
					-1.06295	0.0486846642763938\\
					-1.02265	0.397015083768608\\
					-0.98235	0.109556790198001\\
					-0.94205	0.0493111456306546\\
					-0.90175	0.215856457097866\\
					-0.86145	0.105121582167481\\
					-0.82115	0.292410164458216\\
					-0.78085	0.201764711145779\\
					-0.74055	0.26463910555924\\
					-0.70025	0.19295485938072\\
					-0.65995	-0.0135193065148531\\
					-0.61965	0.148697945639009\\
					-0.57935	0.0318137602911322\\
					-0.53905	0.139347162506438\\
					-0.49875	0.18331925462106\\
					-0.45845	0.163589972295449\\
					-0.41815	-0.0638781834105947\\
					-0.37785	0.148405181718499\\
					-0.33755	0.080781568847368\\
					-0.29725	0.0929423344542004\\
					-0.25695	-0.0310568255558187\\
					-0.21665	-0.00345839513167136\\
					-0.17635	0.175682006591836\\
					-0.13605	0.00262767147814643\\
					-0.0957499999999996	-0.0189662992081097\\
					-0.0554499999999996	-0.0330381369993095\\
					-0.0151499999999996	0.104306688301243\\
					0.0251500000000002	-0.0386750289717742\\
					0.0654500000000002	-0.0302469591321307\\
					0.10575	-0.0143105503749537\\
					0.14605	0.00915369417660941\\
					0.18635	-0.0317812618158691\\
					0.22665	-0.0651212115645952\\
					0.26695	-0.0188621976782727\\
					0.30725	-0.108813671995267\\
					0.34755	-0.0470835960314907\\
					0.38785	-0.049460116630015\\
					0.42815	-0.187808241404692\\
					0.46845	-0.0302898876224359\\
					0.50875	-0.224392212355653\\
					0.54905	-0.12238319882685\\
					0.58935	0.0600133270992933\\
					0.62965	-0.252690130984587\\
					0.66995	-0.196846811859133\\
					0.71025	-0.132041266970158\\
					0.75055	-0.130014476687242\\
					0.79085	-0.1905294867875\\
					0.83115	-0.24136913609419\\
					0.87145	-0.0663775497473345\\
					0.91175	-0.240883317149355\\
					0.95205	-0.11693163578502\\
					0.99235	-0.269779509017702\\
					1.03265	-0.0513912771320339\\
					1.07295	-0.324755544801132\\
					1.11325	-0.00999644715003164\\
					1.15355	-0.401716963149905\\
					1.19385	-0.132077777039911\\
					1.23415	-0.135739126528294\\
					1.27445	-0.124007377568858\\
					1.31475	-0.1678440505312\\
					1.35505	-0.194149378115837\\
					1.39535	-0.201597306871637\\
					1.43565	-0.0447254036855163\\
					1.47595	-0.225393037047338\\
					1.51625	-0.165764433357622\\
					1.55655	-0.151011814836166\\
					1.59685	-0.177872324770815\\
					1.63715	0.0455700202360796\\
					1.67745	-0.0510792255173035\\
					1.71775	0.0428061349705112\\
					1.75805	-0.0230621798421441\\
					1.79835	0.124705853968417\\
					1.83865	0.219087494821745\\
					1.87895	0.0790671861767347\\
					1.91925	0.275935124820764\\
					1.95955	0.623295265332543\\
					1.99985	1.86062620662201\\
				};
				\addlegendentry{$\bar\nu_n$};
				\addplot [color=mycolor2]
				table[row sep=crcr]{%
					-2	-2\\
					-1.95	-0.560423365820456\\
					-1.9	-0.295373319720541\\
					-1.85	-0.163697700586767\\
					-1.8	-0.0788673198280782\\
					-1.75	-0.0179784629517937\\
					-1.7	0.0282488159314517\\
					-1.65	0.0644742960254553\\
					-1.6	0.0933708999472452\\
					-1.55	0.116618998110069\\
					-1.5	0.135348556933219\\
					-1.45	0.150359785969574\\
					-1.4	0.162243169878748\\
					-1.35	0.171449307026282\\
					-1.3	0.17833179856022\\
					-1.25	0.183174783044986\\
					-1.2	0.186211283417518\\
					-1.15	0.187635824996686\\
					-1.1	0.187613353419443\\
					-1.05	0.186285689227065\\
					-1	0.183776298473931\\
					-0.95	0.180193885036514\\
					-0.9	0.175635141276671\\
					-0.85	0.170186886384145\\
					-0.8	0.163927751858114\\
					-0.75	0.156929527086493\\
					-0.7	0.149258246404656\\
					-0.65	0.140975077179863\\
					-0.6	0.132137053117798\\
					-0.55	0.122797686034214\\
					-0.5	0.113007481411275\\
					-0.45	0.102814377255707\\
					-0.4	0.0922641214793183\\
					-0.35	0.0814005998099325\\
					-0.3	0.0702661238188887\\
					-0.25	0.0589016868135305\\
					-0.2	0.0473471939426637\\
					-0.15	0.0356416717945927\\
					-0.0999999999999999	0.0238234619550633\\
					-0.0499999999999998	0.0119304023814144\\
					0	-0\\
					0.0499999999999998	-0.0119304023814144\\
					0.0999999999999999	-0.0238234619550633\\
					0.15	-0.0356416717945927\\
					0.2	-0.0473471939426637\\
					0.25	-0.0589016868135305\\
					0.3	-0.0702661238188887\\
					0.35	-0.0814005998099325\\
					0.4	-0.0922641214793183\\
					0.45	-0.102814377255707\\
					0.5	-0.113007481411275\\
					0.55	-0.122797686034214\\
					0.6	-0.132137053117798\\
					0.65	-0.140975077179863\\
					0.7	-0.149258246404656\\
					0.75	-0.156929527086493\\
					0.8	-0.163927751858114\\
					0.85	-0.170186886384145\\
					0.9	-0.175635141276671\\
					0.95	-0.180193885036514\\
					1	-0.183776298473931\\
					1.05	-0.186285689227065\\
					1.1	-0.187613353419443\\
					1.15	-0.187635824996686\\
					1.2	-0.186211283417518\\
					1.25	-0.183174783044986\\
					1.3	-0.17833179856022\\
					1.35	-0.171449307026282\\
					1.4	-0.162243169878748\\
					1.45	-0.150359785969574\\
					1.5	-0.135348556933219\\
					1.55	-0.116618998110069\\
					1.6	-0.0933708999472452\\
					1.65	-0.0644742960254553\\
					1.7	-0.0282488159314517\\
					1.75	0.0179784629517937\\
					1.8	0.0788673198280782\\
					1.85	0.163697700586767\\
					1.9	0.295373319720541\\
					1.95	0.560423365820456\\
					2	2\\
				};
				\addlegendentry{$\rho_1$}
			\end{axis}
		\end{tikzpicture} \vspace{5.9mm}
		
	\end{subfigure}

	\caption{Comparison of the empirical eigenvalue distribution under nonstandard scaling with the semicircle law. Left: Histogramm of the empirical eigenvalue measure $\bar \mu_n$ of a rescaled Laguerre ensemble and the density of the semicircle law $\musc$. Right: Bardiagramm of $\bar\nu_n = \sqrt{n \beta/2b_n}(\bar\mu_n-\musc)$, which is the rescaled difference of the histogramm of $\bar\mu_n$ and the value of the density of $\musc$ (evaluated in the middle of each block), and the density of $\rho_1$. The simulation was made for $n=20{.}000, \beta = 1$, $\gamma_n = n^{\frac{7}{4}}$ and $b_n = n^2/\gamma_n$, in which case $\xi = 1$.} 
	\label{fig:Histogramm}
\end{figure}

Let us remark that, if $m$ is the moment sequence of a measure $\sigma_m$, which is absolutely continuous with respect to the semicircle law with a signed density $\frac{\partial \sigma_m}{\partial \musc}$, then Parseval's identity implies 
\begin{align} \label{eq:mdprate}
	\mathcal I(m) = \frac{1}{2} \sum_{k=0}^\infty \left( \int q_k(x) \frac{\partial(\sigma_m-\rho_{\xi})}{\partial \musc}(x)\, d\musc(x) \right)^2 = \frac{1}{2} \int \left( \frac{\partial \sigma_m}{\partial \musc}-\frac{\partial\rho_{\xi}}{\partial \musc} \right)^2 \, d\musc .
\end{align}
In this case, the density is uniquely determined by the moment sequence $m$ since the measure $\sigma_m$ is compactly supported and then uniquely determined by $m$. 

It is a characteristic feature of a moderate deviation principle that it holds for a range of rescalings with a common rate function. The rate function in Theorem \ref{thm:MainMDP} does depend on the sequence $(b_n)_n$. We still call the statement a moderate deviation priciple since we allow different orders of rescaling and the possible rate functions belong to a common class depending on $\xi$. Under the additional assumption $n^2/\gamma_n \ll b_n \ll n$, we have $\xi=0$ and then the rate function simplifies to 
\begin{align*}
	\mathcal I(m) = \frac{1}{2} \sum_{k=0}^\infty \left( \int q_k(x) \, d\sigma_m(x) \right)^2 .
\end{align*}
Additionally, if $m$ is the moment sequence of a measure $\sigma_m$, which is absolutely continuous with respect to $\musc$, we have as in \eqref{eq:mdprate} that $\mathcal I (m)$ is given by the squared norm of $\frac{\partial \sigma_m}{\partial \musc}$ in $L^2(\musc)$. In this case, the rate function agrees with the classical one in the MDP for empirical measures in \cite{acosta1994projective}. 

\subsection{Nonstandard Central Limit Theorem for the Laguerre Ensemble}

Our final result is a central limit theorem (CLT)  for the random moments of $\mu_n$ which is formulated as a CLT for $\mu_n$ with polynomial test functions. In the limit appears again a contribution of the signed measure $\rho_{\xi}$ with moments as in \eqref{eq:strangemoments}. 

\begin{thm}
	\label{thm:CltPolynoms}
	Suppose the parameter of the Laguerre ensemble satisfies
	\begin{align*}
		\frac{n\beta}{2\sqrt{\gamma_n}} \xrightarrow[n \to \infty]{} \zeta \in [0, \infty).
	\end{align*}
	Let $\mu_n$ be the spectral measure of the rescaled eigenvalues as in \eqref{eq:spectralmeasure}, 
	then for all polynomials $p$ we have the convergence in distribution
	\begin{align*}
		\sqrt{n\beta/2}\left( \int p \, d\mu_n- \int p \, d \musc\right) \xrightarrow[n\to \infty]{d} \mathcal{N}\left(\mu(p)
		,  \sigma^2(p)
		\right) 
	\end{align*} 
	where     
	\begin{align*}
		\mu(p) = \int p\, d\rho_{\zeta},\qquad   \sigma^2(p) = \int \left( p-\int p \, d\musc\right)^2\,d\musc .
	\end{align*}
\end{thm}

The assumption on the sequence $\gamma_n$ in the CLT is stronger than the one in the MDP. If the limit $\zeta\in [0,\infty)$ exists, we have necessarily $\xi=0$. Conversely, if the MDP holds with $\xi>0$, the CLT will not hold.

\subsection{On an alternative scaling of the eigenvalues} 
\label{sec:AltScaling}

The appearance of the measure $\rho_\xi$ in the MDP and the CLT is a feature unique to the nonstandard scaling $\gamma_n \gg n$. A measure such as $\rho_\xi$ is not present in the limit theorems for the Laguerre ensemble in the regime $\gamma_n \sim cn$ or for the Gaussian ensemble. 

It is tempting to search for a finer rescaling of eigenvalues, such that the measure $\rho_\xi$ does not appear. Looking at the tridiagonal matrix model $\mathcal{L}_n(\beta,\gamma)$ for the Laguerre ensemble (defined in  \eqref{MatrixEntriesLaguerre}), we 
need a sufficiently fast convergence of the rescaled eigenvalues of $\mathcal{L}_n(\beta,\gamma)$ to the semicircle law. The diagonal entries of $\mathcal{L}_n(\beta,\gamma)$ have expectation $n\beta$ (for the first entry) or expectation $2\gamma_n +n\beta$ (for all other diagonal entries, up to a constant). Therefore, we may consider the rescaling 
\begin{align*}
	\hat \lambda_i=   \frac{1}{\sqrt{2\gamma_n n \beta}}(\lambda_i -2\gamma_n-n\beta)
\end{align*}
instead of \eqref{eq:ScalingLDP}. Analogously to \eqref{eq:spectralmeasure}, we define the spectral measure with these rescaled eigenvalues as
\begin{align} \label{eq:altspectralmeasure}
	\hat\mu_n = \sum_{i=1}^n w_i \delta_{\hat \lambda_i} . 
\end{align}
The statements of the LDP in Theorem \ref{thm:LDPLaguerre}, the MDP for $\xi=0$ in Theorem \ref{thm:MainMDP} and the CLT in Theorem \ref{thm:CltPolynoms} for $\zeta=0$ hold then with $\hat\mu_n$ in place of $\mu_n$, since $n\beta$ grows too slowly to have an impact in those cases. 
However, in the MDP with $\xi >0$ now another measure appears: the signed measure $\rho_\xi$ has to be replaced by the measure $\hat \rho_\xi$ which has 
the signed Lebesgue density
\begin{align}
	\label{eq:strangemomentsdensity2}
	-\frac{\xi}{2\pi} \,x\sqrt{4-x^2} \bbbone_{[-2,2]}(x).
\end{align}
The moments of $\hat \rho_\xi$ are given by 
\begin{align} \label{eq:strangemoments2}
	m_k(\hat \rho_\xi) = \begin{cases}
		\xi \left[ \binom{k}{\frac{k-3}{2}}-\binom{k}{\frac{k-1}{2}}\right] & k\text{ odd and } k \geq 3 , \\
		\;0 & k \text{ even or } k=1 .  
	\end{cases} 
\end{align}
We then have the following version of Theorem \ref{thm:MainMDP} for the alternatively rescaled eigenvalues. As in this theorem, the measure $\sigma_m$ denotes any signed measure with moment sequence $m$.

\begin{thm}
	\label{thm:altMainMDP}
	Suppose	$(b_n)_{n\geq 0}$ is a sequence with $b_n\to \infty$, $b_n/n\to 0$ and 
	\begin{align*}
		\frac{n\beta}{2\sqrt{b_n\gamma_n}} \xrightarrow[n \to \infty]{} \xi \in [0,\infty).
	\end{align*}
	Then the sequence of moments sequences $m(\hat\nu_n)$ of the signed measure $\hat\nu_n = \sqrt{n \beta/2b_n}(\hat\mu_n-\musc)$ 
	satisfies a large deviation principle with speed $b_n$ and good rate function    
	\begin{align*}
		\mathcal I(m) =  \frac{1}{2}\sum_{k=0}^\infty \left( \int q_k(x) \, d(\sigma_m-\hat\rho_\xi)(x) \right)^2 , 
	\end{align*}
	where $q_k, k\geq 0$ are the normalized polynomials orthogonal with respect to $\musc$. 
\end{thm} 

The rate function in Theorem \ref{thm:altMainMDP} is zero if and only if $m=m(\hat\rho_\xi)$ which shows that, under our nonstandard scaling, the moments of $\hat\nu_n$ converge to the moments of $\hat\rho_\xi$. In particular, the alternative rescaling does not lead to a convergence of the moments to 0. 
We comment on the necessary modifications to the proof of the MDP in Section \ref{sec:AltScalingProof}. 

A similar statement holds for the CLT in Theorem \ref{thm:CltPolynoms} when $\zeta>0$: we may replace $\mu_n$ by $\hat\mu_n$, but then need to change $\rho_{\zeta}$ to $\hat\rho_{\zeta}$. This statement follows similarly to the alternative version of the MDP when the changes in Section  \ref{sec:AltScalingProof} are applied and we therefore omit commenting on them. \\

\section{Tridiagonal representations}
\label{sec:prelim}

\subsection{Jacobi matrices and spectral measures}

For a sequence $c_k>0,d_k\in \mathbb{R}, k\geq 1$ with $\sup_k\{c_k+|d_k|\}<\infty$, we define the infinite Jacobi matrix 
\begin{align}\label{eq:tridigonalinfinite}
	\mathcal{J}  = \begin{pmatrix}
		d_1 & c_1 &  \\
		c_1 & d_2 & c_2 \\
		& c_2 & \ddots & \ddots\\
		& & \ddots &  
	\end{pmatrix} .
\end{align}
Then $\mathcal J$ is a self-adjoint bounded operator on $\ell^2$ with cyclic vector the first unit vector $e_1$. By the spectral theorem, there exists a unique spectral measure $\mu$ of the pair $(\mathcal J,e_1)$, which may be defined by the moment relation
\begin{align}\label{eq:spectralmeasuremoments}
	\int x^k \, d\mu(x) = \langle e_1, \mathcal J^k e_1 \rangle 
\end{align}
for $k\geq 0$. The support of $\mu$ is given by the spectrum of $\mathcal J$ and is therefore infinite, but compact. 

On the other hand, suppose $\mu$ is a probability measure on $\mathbb R$ with compact but infinite support. Applying the Gram-Schmidt procedure
to the sequence $1, x, x^2, \dots$ in $L^2(\mu)$, we obtain a sequence $p_0,p_1,\dots $ of orthonormal polynomials with positive leading coefficients. 
They obey the recursion relation
\begin{align} \label{eq:polrecursion}
	xp_k(x) = c_{k+1} p_{k+1}(x) + d_{k+1} p_k (x) + c_{k} p_{k-1}(x)
\end{align}
for $ k \geq 0$, where the recursion coefficients satisfy $d_k \in \mathbb R, c_k > 0$ for all $k\geq 1$ and with $p_{-1}(x)=0$.  
As a consequence of this recursion, the linear transformation $f(x) \rightarrow xf(x)$ is represented in the basis $\{p_0, p_1, \dots\}$  by the matrix $\mathcal J$ as in \eqref{eq:tridigonalinfinite}. 
The mapping $\mu \mapsto \mathcal J$ is a one to one correspondence between probability measures on $\mathbb R$ having compact infinite support and Jacobi matrices $\mathcal J$, which is also called the Szeg\H{o} mapping.

An important role is played by the so-called free Jacobi matrix $\mathcal J_0$, with coefficients
\begin{align}
	\label{eq:freeJacobi}
	d_k = 0, \qquad c_k=1
\end{align}
for all $k\geq 1$, whose spectral measure is the semicircle law $\mu_{sc}$. 

Starting from a finite $n\times n$ tridiagonal matrix
\begin{align}\label{eq:tridiagonalfinite}
	\mathcal{J}_n =\begin{pmatrix}
		d_1 & c_1 &  \\
		c_1 & d_2 & \ddots \\
		& \ddots & \ddots & c_{n-1}\\
		& & c_{n-1} & d_n
	\end{pmatrix} 
\end{align}
with positive subdiagonal, we can define the spectral measure $\mu_n$ of $\mathcal J_n$ by the analogous moment relation as in \eqref{eq:spectralmeasuremoments}, that is, 
\begin{align}\label{eq:spectralmeasuremomentsfinite}
	\int x^k \, d\mu_n(x) = \langle e_1, \mathcal J_n^k e_1 \rangle . 
\end{align}
A straightforward calculation reveals that it is the finitely supported measure 
\begin{align}
	\label{eq:spectralmeasure2}
	\mu_n = \sum_{i=1}^n w_i \delta_{\lambda_i} , 
\end{align}
with $\lambda_1,\dots , \lambda_n$ the eigenvalues of $\mathcal J_n$ and with weights $w_i = \langle u_i,e_1\rangle^2$, for $u_1,\dots , u_n$ a corresponding orthonormal basis of eigenvectors. Compared with the empirical eigenvalue measure $\bar\mu_n$, it is a weighted version with the weights containing information on the eigenvectors. 

As in the infinite dimensional case, we may start with a finitely supported measure as in \eqref{eq:spectralmeasure2} and define the first $n-1$ polynomials orthonormal with respect to $\mu_n$. They satisfy then the recursion \eqref{eq:polrecursion} with coefficients given as entries in the matrix \eqref{eq:tridiagonalfinite}.

\subsection{Random spectral measures}
\label{sec:RandTriMod}


The tridiagonal representation of the Laguerre ensemble was constructed by \cite{Dumitriu_2002}, extending the work of \cite{trotter1984}. It yields a random matrix model with eigenvalue density given by \eqref{DensityLaguerre} for general parameter $\beta>0$. We denote by $\chi^2_r$ the $\chi^2$-distribution with $r >0$ degrees of freedom which has the probability density function
\begin{align*}
	\frac{1}{2^{r/2}\Gamma(r/2)} x^{r/2-1}e^{-x/2} \bbbone_{[0,\infty)}(x) ,
\end{align*}
where $\Gamma$ denotes the gamma function.
Let $n\in \mathbb N$, $\beta >0$ and $\gamma>(n-1)\beta/2$ and suppose 
$z_1,...,z_{2n-1}$ are independent random variables with 
\begin{align*}
	z_k \sim \begin{cases}
		\chi^2_{2\gamma-\beta(k-1)/2} & \text{for $k$ odd,}\\
		\chi^2_{\beta(2n-k)/2} & \text{for $k$ even}.
	\end{cases}    
\end{align*}
Let $\mathcal{L}_n(\beta,\gamma)$ denote the random tridiagonal matrix as in \eqref{eq:tridiagonalfinite} with entries given by 
\begin{equation}
	\label{MatrixEntriesLaguerre}
	\begin{aligned}
		d_k &= z_{2k-1} + z_{2k-2} \\
		c_k^2 &= z_{2k-1} z_{2k},
	\end{aligned}
\end{equation}
with $z_0 = 0$. Then the distribution of the spectral measure $\mu_n$ of $\mathcal{L}_n(\beta,\gamma)$ is characterized by the following: the eigenvalues have the Lebesgue density \eqref{DensityLaguerre}, the eigenvalues and weights are independent and the weights are Dirichlet distributed on the standard simplex with homogeneous parameter $\beta/2$ defined by \eqref{eq:dirdensity}. 

The tridiagonal representation $\mathcal{L}_n(\beta,\gamma)$ is particularly convenient to study limit theorems for the spectral measure $\mu_n$. To illustrate, let us consider the classical law of large numbers for the Laguerre ensemble. Suppose $\gamma=\gamma_n$ is such that $2\gamma_n/n\beta\to \tau^{-1}\in [1,\infty)$, then it is easy to see that the rescaled matrix $(n\beta)^{-1}\mathcal{L}_n(\beta,\gamma_n)$ converges in probability entrywise to the infinite tridiagonal matrix with entries
\begin{align*}
	\bar c_k = \sqrt \tau \ (k\geq 1), \qquad  \bar d_1 = 1  , \quad \bar d_k = 1+\tau \ (k \geq 2) . 
\end{align*}
The spectral measure of this matrix is given by the Marchenko-Pastur law $\mu_{\mathrm{MP}(\tau)}$ with density
\begin{align}
	\frac{\sqrt{(\tau^+-x)(x-\tau^-)}}{2\pi \tau x} \mathbbm{1}_{\{\tau^-<x<\tau^+\}} , 
\end{align}
with $\tau^\pm = (1\pm \sqrt{\tau})^2$. Convergence of the recursion coefficients implies that the moments of the spectral measure of $(n\beta)^{-1}\mathcal{L}_n(\beta,\gamma_n)$ converge to the moments of $\mu_{\mathrm{MP}(\tau)}$. Since this measure is compactly supported, it is uniquely determined by its moments and then the convergence of moments implies the weak convergence of the spectral measure to $\mu_{\mathrm{MP}(\tau)}$ in probability. This convergence could also be transferred to the empirical eigenvalue measure, since the Kolmogorov distance between $\mu_n$ and $\bar\mu_n$ vanishes as $n\to \infty$ \cite[Lemma 4.2]{nagel2014nonstandard}.

\section{Proofs}
\label{sec:proofs}

\subsection{Overview of the proof strategy}

The ideas of the proofs and the overall technique are based on \cite{gamboa2011large} and \cite{nagel2014nonstandard}. Instead of working with the distribution of support points and weights of the measure $\mu_n$ directly, we start with the entries of the Jacobi matrix of $\mu_n$, which have explicit distributions and sufficient independence structure. 

The measure $\mu_n$ with support points $\tilde \lambda_i$ as in \eqref{eq:ScalingLDP} is the spectral measure of the rescaled tridiagonal matrix 
\begin{align} \label{eq:Lrescaled}
	\tilde{\mathcal{L}}_n(\beta,\gamma) := \frac{1}{\sqrt{2\gamma_nn\beta}}\left( \mathcal{L}_n(\beta,\gamma) - 2\gamma_n I_n\right) ,
\end{align}
where $\mathcal{L}_n(\beta,\gamma)$ is the tridiagonal matrix of the Laguerre ensemble defined by \eqref{MatrixEntriesLaguerre} and $I_n$ is the $n\times n$ identity matrix. Writing $\tilde d_k, \tilde c_k$ for the entries of the tridiagonal matrix $\tilde{\mathcal{L}}_n(\beta,\gamma)$, we have 
\begin{align*}
	\tilde{d}_k = \frac{1}{\sqrt{2\gamma_nn\beta}} z_{2k-2}+  \frac{1}{\sqrt{2\gamma_nn\beta}}( z_{2k-1} -2\gamma_n) 
\end{align*}
for $1\leq k\leq n$, with $z_0 = 0$. The off-diagonal elements are the square root of
\begin{align*}
	\tilde{c}^2_k:=\left( \frac{1}{2 \gamma_n}z_{2k-1}\right)\left(\frac{1}{ n\beta} z_{2k} \right).
\end{align*}
for $1\leq k\leq n-1$. 
Note that the entries of $r^{(n)}=(\tilde{d}_1,\tilde{c}_1,\dots)$ are not independent. However, the random variables $z_1,z_2,\dots$ are independent. The proofs of the LDP, the MDP and the CLT therefore start by showing corresponding statements for a finite collection of $z_k$, and then for a finite collection of entries of $r^{(n)}$. The main challenge in the proofs is to extend limit theorems for finitely many entries of $r^{(n)}$ to the complete infinite vector, and then to transfer these statements to the spectral measure. We comment on these difficulties at the beginning of the following sections.

\subsection{Proof of Theorem \ref{thm:LDPLaguerre} (LDP)}

In the proof, we will show an LDP for a finite projection of the vector $(z_1,z_2,\dots)$ and use the projective method of the Dawson-Gärtner theorem to expand the LDP to the infinite vector, making use of the independence of the entries. The contraction principle yields the LDP for the infinite vector $(r^{(n)})_n$. We then have to restrict the set of the LDP to define a continuous mapping to the spectral measure, so that we can use the contraction principle once again to obtain the LDP for $(\mu_n)_n$. 

We start the proof with an auxiliary result for sequences of $\chi^2$-distributed random variables.
\begin{lem}
	\label{LDPLemma}
	Let $X_n \sim \chi^2_{\alpha_n +\alpha_0}$ and $Y_n \sim \chi^2_{\beta_n +\beta_0}$ be independent random variables where $\alpha_0,\beta_0\in \RR$ and let $\alpha_n$ and $\beta_n$ be positive real numbers such that $\alpha_n,\beta_n \to \infty$ and
	\begin{align*}
		\lim_{n \to \infty} \frac{\alpha_n}{\beta_n} = \infty.
	\end{align*}
	Then the vector
	\begin{align*}
		\mathcal{G}^{(n)} = \left(\frac{1}{\sqrt{\alpha_n\beta_n}}(X_n -\alpha_n),\frac{1}{\alpha_n}X_n,\frac{1}{\sqrt{\alpha_n\beta_n}} Y_n,\frac{1}{\beta_n} Y_n \right)
	\end{align*}
	satisfies an LDP in $\RR \times [0,\infty)^3$ with speed $\beta_n/2$. The good rate function $\mathcal{I}_0$ is given by
	\begin{align*}
		\mathcal{I}_0(x) = \frac{1}{2}x_1^2+x_4-1 -\log x_4
	\end{align*}
	if $x_2 = 1$, $x_3 = 0$ and $x_4 >0$, and $\mathcal{I}_0(x) = +\infty$ otherwise.
	\begin{proof}
		The moment-generating function of $X_n$ is
		\begin{align*}
			\EE\big[\exp(sX_n)\big]= (1-2s)^{-\frac{1}{2}(\alpha_n+\alpha_0)}, \quad s < \tfrac{1}{2}.
		\end{align*}
		Now, let $t \in \RR^3 \times (-\infty,1)$ and consider for $n$ large enough
		\begin{align*}
			2\beta_n^{-1} \log \EE &\exp\left\{\frac{\beta_n}{2} \langle t, \mathcal{G}^{(n)} \rangle \right\}
			=2\beta_n^{-1} \log \EE\left[ \exp\left\{ \sqrt{\frac{\beta_n}{\alpha_n}}\frac{t_1}{2}(X_n-\alpha_n)  + \frac{\beta_n}{\alpha_n}\frac{t_2}{2} X_n  + \sqrt{\frac{\beta_n}{\alpha_n}}\frac{t_3}{2} Y_n + \frac{t_4}{2} Y_n\right\} \right]   \\[0,2cm]
			&= 2\beta_n^{-1} \log \EE\left[ \exp\left\{ \Big(\sqrt{ \frac{\beta_n}{\alpha_n}} \frac{t_1}{2} + \frac{\beta_n}{\alpha_n}\frac{t_2}{2}\Big)X_n -\sqrt{\beta_n \alpha_n}\frac{t_1}{2} + \Big( \sqrt{\frac{\beta_n}{\alpha_n}} \frac{t_3}{2} + \frac{t_4}{2} \Big) Y_n\right\}\right ] \\[0,2cm]
			&= \beta_n^{-1}  \log\Big(1-\sqrt{ \frac{\beta_n}{\alpha_n}} t_1 - \frac{\beta_n}{\alpha_n} t_2\Big)^{-\alpha_n-\alpha_0} -\sqrt{\frac{\alpha_n}{\beta_n}}t_1 + \beta_n^{-1} \log\Big( 1 -\sqrt{\frac{\beta_n}{\alpha_n}} t_3 - t_4  \Big)^{-\beta_n-\beta_0} .
		\end{align*}
		Now, a Taylor expansion of the logarithm
		yields for the first two summands
		\begin{align*}
			-\frac{\alpha_n + \alpha_0}{\beta_n}  \log\Big(1-\sqrt{ \frac{\beta_n}{\alpha_n}} t_1 - \frac{\beta_n}{\alpha_n} t_2\Big) -\sqrt{\frac{\alpha_n}{\beta_n}}t_1 
			&= \frac{\alpha_0}{\alpha_n} \sqrt{\frac{\alpha_n}{\beta_n}}t_1 + \frac{\alpha_n+\alpha_0}{\alpha_n} \left( t_2 + \frac{1}{2}t_1^2 \right)+ \mathcal{O}\Big( \sqrt{\frac{\beta_n}{\alpha_n}}\Big)
			\\ &\xrightarrow[n \to \infty]{} \frac{1}{2}t_1^2+t_2 ,
		\end{align*}
		and the last summand converges to $-\log(1-t_4)$.
		The function 
		\begin{align*}
			\mathcal{C}(t) := \lim_{n\to \infty} 2\beta_n^{-1} \log \EE &\exp\left\{\frac{\beta_n}{2} \langle t, \mathcal{G}^{(n)} \rangle \right\} = \frac{1}{2}t_1^2+t_2- \log(1-t_4)
		\end{align*}
		is finite and differentiable on $\RR^3 \times (-\infty,1)$. Hence, it is possible to apply the Gärtner-Ellis Theorem, see \cite[Theorem 2.3.6]{dembo2009large}, yielding the LDP with speed $\beta_n/2$ and good rate function 
		\begin{align*}
			\mathcal{I}_0(x)= \sup_{t \in \RR^4}\left\{\langle t,x \rangle - \mathcal{C}(t)\right\},
		\end{align*}
		which is infinite unless $x_2 = 1$, $x_3 = 0$ and $x_4 >0$. In this case
		\begin{align*}
			\mathcal{I}_0(x) = \frac{1}{2}x_1^2+x_4-1 -\log x_4.
		\end{align*}
	\end{proof}
\end{lem}

Applying
Lemma \ref{LDPLemma} with $\alpha_n = 2\gamma_n$ and $\beta_n = n \beta$ gives an LDP for 
\begin{align*}
	\left(\frac{1}{\sqrt{2\gamma_nn\beta}}( z_{2k-1} -2\gamma_n),\frac{1}{2 \gamma_n}z_{2k-1},\frac{1}{\sqrt{2\gamma_nn\beta}} z_{2k},\frac{1}{ n\beta} z_{2k} \right)
\end{align*}
in $\RR \times [0,\infty)^3$ with speed $n \beta/2$ and good rate function
\begin{align*}
	\mathcal{I}_1(x) = \frac{1}{2}x_1^2+x_4-1-\log (x_4)
\end{align*}
if $x_2 = 1$, $x_3 = 0$ and $x_4 >0$.
Define for $1\leq m \leq 2n-1$
\begin{align*}
	y_m^{(n)} = \begin{cases}
		\left( \frac{1}{\sqrt{2\gamma_n n\beta}}(z_{2k-1}-2\gamma_n), \frac{1}{2\gamma_n}z_{2k-1}\right), & m = 2k-1\\[2mm]
		\left( \frac{1}{\sqrt{2\gamma_n n\beta}}z_{2k},\frac{1}{n\beta}z_{2k}\right), & m = 2k 
	\end{cases}
\end{align*}
and set $y_m^{(n)}=0$ if $m\geq 2n-1$. We let $y_{m,1}$ denote the first and $y_{m,2}$ the second entry of $y_m$.  
Because $z_1,z_2,...,z_{2n-1}$ are independent, \cite[Exercise 4.2.7]{dembo2009large} yields that for any $K\geq 1$, the sequence $(y_1^{(n)}, \dots ,y_{2K}^{(n)})_n$ satisfies the LDP in $(\RR \times [0,\infty)^3)^K$ with speed $n\beta/2$ and good rate function 
\begin{align*}
	\mathcal I_2(y_1,\dots ,y_{2K}) = \sum_{k=1}^K \mathcal{I}_1(y_{2k-1,1},y_{2k-1,2},y_{2k,1},y_{2k,2}) . 
\end{align*}
We now apply the projective method of the Dawson-Gärtner Theorem, see \cite[Theorem 4.6.1]{dembo2009large}. It yields the LDP for the sequence of infinite vectors
\begin{align*}
	y^{(n)}:= (y_{1,1},y_{1,2},y_{2,1},y_{2,2},...) :=
	\left(  \frac{1}{\sqrt{2\gamma_n n\beta}}(z_{1}-2\gamma_n), \frac{1}{2\gamma_n}z_{1},\frac{1}{\sqrt{2\gamma_n n\beta}}z_{2},\frac{1}{n\beta}z_{2},...,\frac{1}{2\gamma_n}z_{2n-1}, 0 ,\dots  \right)
\end{align*}
in $\big(\RR\times [0,\infty)^3\big)^\NN$ with speed $n\beta/2$ and good rate function
\begin{align*}
	\mathcal{I}_3(y) = \sum_{k=1}^\infty \mathcal{I}_1(y_{2k-1,1},y_{2k-1,2},y_{2k,1},y_{2k,2}) .
\end{align*}
This rate function is infinite unless $y_{2k-1,2} = 1$, $y_{2k,1} = 0$ and $y_{2k,2} >0$ for all $k \in \NN$ and in this case has the form
\begin{align*}
	\mathcal{I}_3(y) = \sum_{k=1}^\infty \frac{1}{2}y^2_{2k-1,1} + y_{2k,2} -1 -\log y_{2k,2}.
\end{align*}
Now, let $r^{(n)}=(\tilde{d}_1,\tilde{c}_1,\tilde{d}_2,\tilde{c}_2,\dots , \tilde d_n,0,...)$ denote the entries of $\tilde{\mathcal{L}}_n(\beta,\gamma)$ as in \eqref{eq:Lrescaled}, continued by zeros. Then $r^{(n)}$ is a continuous function of $y^{(n)}$. It can be formalized by
\begin{align*}
	\Phi: \RR^\NN \to \RR^\NN, \quad \Phi_1(x) = x_1, \quad \Phi_{2k}(x) = x_{4k-2}x_{4k}, \quad \Phi_{2k+1}(x) = x_{4k-1} + x_{4k+1}, \quad \text{for } k \in \NN  
\end{align*}
such that $r^{(n)} = \Phi(y^{(n)})$. The contraction principle yields an LDP for $(r^{(n)})_n$ in $\{x\in \RR^\NN \mid x_{2k}\geq 0  \text{ for } k \in \NN\} $, endowed with the product topology, with speed $n\beta/2$ and good rate function 
\begin{align*}
	\mathcal{I}_4(d_1,c_1,d_2,c_2,...) = \sum_{k=1}^\infty \frac{1}{2}d_k^2+c_k^2-1-\log c_k^2
\end{align*}
if $c_{k}>0$ for all $k\geq 1$ and $\mathcal{I}_4(d_1,c_1,d_2,c_2,...)=\infty$ otherwise. It remains to map the sequence of coefficients to the spectral measure. 
To obtain a continuous mapping to the set of probability measures on $\RR$ with compact support, we first restrict the LDP to the set
\begin{align*} 
	\mathcal{E} :=  \bigg(\big\{ x \mid x_{2k}> 0  \big\} \cup \bigcup_{n=1}^\infty \{(x_1,...,x_{2n},0,...), \;\; x_{2k}>0 \; \forall k \in\{1,...,n\}  \big\}\bigg) \cap \{x \mid \sup_k |x_k| < \infty\}
\end{align*}
which is equipped with the subspace topology. All $x \in \{x\in \RR^\NN \mid x_{2k}\geq 0  \text{ for } k \in \NN\} $ for which $\mathcal{I}_4(x)<\infty$ are elements of $\mathcal{E}$, hence \cite[Lemma 4.1.5]{dembo2009large} may be applied and we obtain that $(r^{(n)})_n$ satisfies the LDP in $\mathcal{E}$.
Finally, the mapping
\begin{align*}
	\varphi: \mathcal{E} \to \mathcal{M}_1^c:=\{ \mu \mid \mu \text{ probability measure on } \RR \text{ with compact support}\},
\end{align*}
which maps a sequence of Jacobi coefficients to its spectral measure, is a bijection, see \cite[Chapter 2.1]{deift2000orthogonal} for the construction and more information. When $\mathcal{M}_1^c$ is endowed with the weak topology, $\varphi$ is continuous: suppose $x_n \to x$ in $\mathcal E$, then the Jacobi matrix build from $x_n$ converges entrywise to the Jacobi matrix build from $x$, such that, according to \eqref{eq:spectralmeasuremoments}, 
\begin{align*}
	m_k(\varphi(x_n)) \xrightarrow[n\to \infty]{} m_k(\varphi(x)) 
\end{align*}
for all $k\geq 1$. This yields the continuity of $\varphi$ because convergence of moments yields weak convergence if $\varphi(x)$ is compactly supported. We have $\mu_n= \varphi(r^{(n)})$, so that by the contraction principle, the sequence $(\mu_n)_n$ satisfies an LDP in $\mathcal{M}_1^c$ with speed $n\beta/2$ and good rate function
\begin{align*}
	\mathcal{I}(\mu) = \inf_{x \in \mathcal{E}: \varphi(x) = \mu } \mathcal{I}_4(x).
\end{align*}
This infimum is infinite unless $\mu$ is nontrivial (i.e. the support of $\mu$ is not a finite set of points) and since $\varphi$ is bijective, it is given by $\mathcal{I}(\mu) = \mathcal{I}_4(\varphi^{-1}(\mu))$. By the sum rule of Killip and Simon \cite{killip2003sum}, $\mathcal I$ is then as given in \eqref{eq:sumrule}.

\subsection{Proof of Theorem \ref{thm:MainMDP} (MDP)}
\label{section:MDP}

The technique used is similar to the one in \cite{nagel2013moderatedeviationsspectralmeasures}. 
We begin by proving an MDP for a finite projection of $(z_1,z_2,...)$ in the same way as in the LDP proof. To transfer this to the entries of the Jacobi matrix, we utilize in Lemma \ref{MomentsMDP} the Delta method for large deviations. This requires us to work with finite vectors only, so that we have to apply the projective Dawson-Gärtner method afterwards. We also do not map to the set of compact probability measures, like we did in the LDP, but instead to the set of moments of signed measures. 
In the remaining section, we prove an integral representation of the rate function and apply the Gärtner-Ellis theorem to obtain the MDP for the infinite vector of moments.
Similar to the previous section, the first lemma is an auxiliary result for sequences of $\chi^2$-distributed random variables.
\begin{lem}
	\label{MDPLemma}
	Let $X_n \sim \chi^2_{\alpha_n+\alpha_0} $ and $Y_n \sim \chi^2_{{\beta}_n+\beta_0}$ be independent random variables where $\alpha_0,\beta_0 \in \RR$ and $\alpha_n$ and $\beta_n$ are positive real numbers such that $\alpha_n, \beta_n \to \infty$ and
	\begin{align*}
		\lim_{n \to \infty} \frac{\alpha_n}{\beta_n} = \infty.
	\end{align*}
	Additionally, let $(b_n)_n$ be a real sequence such that $b_n \to \infty$,
	\begin{align*}
		\lim_{n \to \infty}\frac{b_n}{\alpha_n} = 0\quad \text{and} \quad \lim_{n \to \infty}\frac{\beta_n}{\sqrt{2b_n \alpha_n}} = \xi \in [0,\infty).
	\end{align*}
	Then the vector
	\begin{align*}
		\mathcal{G}^{(n)} = \sqrt{\frac{\beta_n}{2b_n}} \left(\frac{1}{\sqrt{\alpha_n\beta_n}}(X_n -\alpha_n),\frac{1}{\alpha_n}X_n-1,\frac{1}{\sqrt{\alpha_n\beta_n}} Y_n,\frac{1}{\beta_n} Y_n -1\right)
	\end{align*}
	satisfies an LDP in $\RR \times [-1,\infty) \times \RR \times [-1,\infty)$ with speed $b_n$ and good rate function $\mathcal{I}_0$ given by
	\begin{align*}
		\mathcal{I}_0(x) = \frac{1}{2}(x_1^2+x_4^2)
	\end{align*}
	if $x_2=0$ and $x_3 = \xi$, and $\mathcal{I}_0(x) = +\infty$ otherwise.
	\begin{proof}
		As in Lemma \ref{LDPLemma}, we apply the Gärtner-Ellis Theorem \cite[Theorem 2.3.6]{dembo2009large}. For that consider for $t \in \RR^4$ and $n$ large enough
		\begin{align*}
			b_n^{-1}\log \EE&\big[\exp\big(b_n \langle t, \mathcal{G}^{(n)} \rangle\big)\big] \\
			=&b_n^{-1} \log \EE\Bigg[\exp\left(\left(\sqrt{\frac{b_n}{\alpha_n}}\frac{t_1}{\sqrt{2}}+\frac{\sqrt{b_n \beta_n}}{\alpha_n} \frac{t_2}{\sqrt{2}}\right)X_n- \sqrt{b_n\alpha_n}\frac{t_1}{\sqrt{2}} - \sqrt{b_n\beta_n}\frac{t_2}{\sqrt{2}}\right)\Bigg]\\
			+& b_n^{-1} \log \EE\Bigg[\exp\left(\left(\sqrt{\frac{b_n}{\alpha_n}}\frac{t_3}{\sqrt{2}}+\sqrt{\frac{b_n }{\beta_n}}\frac{t_4}{\sqrt{2}}\right)Y_n- \sqrt{b_n\beta_n}\frac{t_4}{\sqrt{2}}\right)\Bigg] \\
			= &b_n^{-1} \log\left(1-\sqrt{\frac{2b_n}{\alpha_n}}t_1-\frac{\sqrt{2b_n \beta_n}}{\alpha_n}t_2\right)^{-\frac{1}{2}({\alpha}_n+\alpha_0)} - \sqrt{\frac{\alpha_n}{2b_n}}t_1 - \sqrt{\frac{\beta_n}{2b_n}}t_2 \\[2mm]
			+& b_n^{-1} \log\left(1-\sqrt{\frac{2b_n}{\alpha_n}}t_3-\sqrt{\frac{2b_n }{\beta_n}}t_4\right)^{-\frac{1}{2}({\beta}_n+\beta_0)} - \sqrt{\frac{\beta_n}{2b_n}}t_4. 
		\end{align*}
		In order to calculate the limit, the Taylor expansion for the logarithm  is used. 
		For the first part, we have 
		\begin{align*}
			-\frac{\alpha_n+\alpha_0}{2b_n}&\log\left(1-\sqrt{\frac{2b_n}{\alpha_n}}t_1 - \frac{\sqrt{2b_n\beta_n}}{\alpha_n}\right)
			- \sqrt{\frac{\alpha_n}{2b_n}}t_1 - \sqrt{\frac{\beta_n}{2b_n}}t_2 \\[2mm]
			&= \frac{1}{2}t_1^2+ \frac{1}{2} \frac{\beta_n}{\alpha_n} t_2^2 + \sqrt{\frac{\beta_n}{\alpha_n}} t_1t_2 + \mathcal{O}\Big( \sqrt{\frac{b_n}{\alpha_n}}\Big)
			\xrightarrow[n \to \infty]{} \frac{1}{2} t_1^2.
		\end{align*}
		For the second part, we obtain
		\begin{align}
			\label{eq:TaylorMDPxi}
			\begin{split}
				-\frac{\beta_n+\beta_0}{2b_n}&\log\left( 1-\sqrt{\frac{2b_n}{\alpha_n}}t_3-\sqrt{\frac{2b_n }{\beta_n}}t_4\right) - \sqrt{\frac{\beta_n}{2b_n}}t_4 \\
				&= \frac{\beta_n}{\sqrt{2b_n \alpha_n}}t_3+\frac{1}{2} \frac{\beta_n}{\alpha_n}t_3^2 +\frac{1}{2}t_4^2 + \sqrt{\frac{\beta_n}{\alpha_n}}t_3t_4 + \mathcal{O}\Big( \sqrt{\frac{b_n}{\alpha_n}}\Big)
				\xrightarrow[n \to \infty]{} \xi t_3 +\frac{1}{2}t_4^2. 
			\end{split}
		\end{align}
		The function
		\begin{align}
			\label{eq:CtMDP}
			\mathcal{C}(t) := \lim_{n\to \infty } b_n^{-1}\log \EE&\big[\exp\big(b_n \langle t, \mathcal{G}^{(n)} \rangle\big)\big] = \frac{1}{2} (t_1^2+t_4^2) + \xi t_3
		\end{align}
		is finite and differentiable on $\RR^4$. Then the Gärtner-Ellis Theorem yields an LDP for $(\mathcal G^{(n)})_n$ with speed $b_n$ and good rate function
		\begin{align*}
			\mathcal{I}_0(x) = \sup_{t \in \RR^4}\{ \langle t, x \rangle- \mathcal{C}(t)\}
		\end{align*}
		which is infinite unless $x_2=0$ and $x_3 = \xi$ and in that case $\mathcal I_0(x)$ is as stated.
	\end{proof}
\end{lem}

In the following lemma the dimension of the moment vector plays an important role since it is easier to apply the Delta method to MDPs on sets with finite dimension. To specify this, we denote the truncation of $x \in \RR^\NN$ by $x_{[K]} = (x_k)_{1\leq k \leq K}$.
In particular, the truncated moment sequence of $\nu \in \mathcal{M}_0^f$ is denoted by $m_{[K]}(\nu) = (m_k(\nu))_{1 \leq k \leq K} \in \RR^K$.

An important part in the MDP is played by the matrix $D\in \RR^{\NN \times \NN}$ given by
\begin{align}
	\label{eq:D}
	D_{i,j} = 
	\binom{i}{\frac{i-j}{2}} - \binom{i}{\frac{i-j}{2}-1} 
\end{align}
if $i \geq j $ and $i+j$ even and $D_{i,j}=0$ otherwise.
The truncated matrix $D_K \in \RR^{K\times K}$ is defined by $(D_K)_{i,j} = D_{i,j}$ for $1 \leq i,j \leq K$. 

\begin{lem}
	\label{MomentsMDP}
	Let $(\gamma_n)_n$ and $(b_n)_n$ be real sequences that satisfy
	\begin{align*}
		\lim_{n \to \infty} \frac{\gamma_n}{n} = \infty ,\quad \lim_{n \to \infty}\frac{b_n}{n} = 0 \quad \text{and} \quad \lim_{n \to \infty}\frac{n\beta/2}{\sqrt{b_n\gamma_n}} = \xi \in [0,\infty).
	\end{align*}
	Let $\mu_n$ denote the spectral measure of the Laguerre ensemble with rescaled eigenvalues as in \eqref{eq:spectralmeasure}
	and $\mu_{\text{sc}}$ denote the semicircle law. 
	Then the sequence of $(2K-1)$-dimensional moment vectors 
	$$\sqrt{\frac{n\beta}{2b_n}}\big(m_{[2K-1]}(\mu_n)-m_{[2K-1]}(\mu_{\text{sc}})\big)$$
	satisfies an LDP in $\RR^{2K-1}$, equipped with the product topology, with speed $b_n$ and good rate function
	\begin{align*}
		\mathcal{I}(m_{[2K-1]}) = \frac{1}{2}\left\Vert D_{2K-1}^{-1} \big(m_{[2K-1]} - D_{2K-1}w_{[2K-1]} \big)  \right\Vert_2^2
	\end{align*}
	where $w = (0,0,\xi,0,\xi,0,\xi,...)^\top$ and $D_{2K-1} \in \RR^{(2K-1) \times (2K-1)}$ as above.
\end{lem}
\begin{proof} 
	The diagonal entries of the rescaled $n$-dimensional Jacobi matrix \eqref{eq:Lrescaled} of $\mu_n$ can be rewritten using (\ref{MatrixEntriesLaguerre}) as
	\begin{align*}
		\Tilde{d}_k:= \frac{1}{\sqrt{2\gamma_nn\beta  }}(z_{2k-1} -2\gamma_n) + \frac{1}{\sqrt{2\gamma_nn\beta  }}z_{2k-2}
	\end{align*}
	for $1\leq k \leq n$, with $z_0=0$. The off-diagonal entries are the squareroot of
	\begin{align*}
		\Tilde{c}_k^2:= \frac{1}{2\gamma_n}z_{2k-1}\frac{1}{n\beta } z_{2k}
	\end{align*}
	for $1 \leq k \leq n-1$. Applying Lemma \ref{MDPLemma} with $\alpha_n = 2\gamma_n$ and $\beta_n = n \beta$ to 
	\begin{align}
		\label{eq:GOriginalMDP}
		\sqrt{\frac{n\beta }{2b_n}} \left(\left( \frac{1}{\sqrt{2\gamma_nn\beta  }}(z_{2k-1} -2\gamma_n),\frac{1}{2\gamma_n}z_{2k-1},\frac{1}{\sqrt{2\gamma_n n\beta  }} z_{2k},\frac{1}{n\beta } z_{2k}\right) - (0,1,0,1)\right)
	\end{align}
	yields that this vector fulfills an LDP in $\RR \times [-1,\infty) \times \RR \times [-1,\infty)$ with speed $b_n$ and good rate function
	\begin{align*}
		\mathcal{I}_1(x) = \frac{1}{2}(x_1^2+x_4^2)
	\end{align*}
	if $x_2=0$ and $x_3 = \xi$. Now, define
	\begin{align}
		\label{ErsteFktDelta}
		\Phi: \RR \times [0,\infty) \times \RR \times [0,\infty) \to \RR\times [0,\infty) \times \RR, \quad \Phi(x_1,x_2,x_3,x_4) = (x_1,x_2\cdot x_4,x_3)
	\end{align}
	which has the Jacobian matrix
	\begin{align*}
		\frac{\partial \Phi(x)}{\partial x} = \left(\begin{matrix} 1 & 0 & 0 &0\\0&x_4&0&x_2 \\0 & 0 & 1 & 0 \end{matrix}\right).
	\end{align*}
	The Delta method for MDPs, see \cite{gao2011delta}, then yields an LDP for
	\begin{align*}
		\sqrt{\frac{n\beta}{2b_n}}\left( \left(\frac{1}{\sqrt{2\gamma_nn\beta  }}(z_{2k-1} -2\gamma_n),\frac{1}{2\gamma_n n\beta }z_{2k-1}z_{2k} ,\frac{1}{\sqrt{2\gamma_n n\beta  }} z_{2k}\right)-(0,1,0)\right)
	\end{align*}
	with speed $b_n$ and good rate function 
	\begin{align*}
		\mathcal{I}_2(y) = \inf\left\{ \mathcal{I}_1(x) : \frac{\partial \Phi(h)}{\partial h}|_{h=(0,1,0,1)} \cdot x = y \right\}
		= \begin{cases}
			\frac{1}{2}(y_1^2+y_2^2) & \text{if } y_3=\xi \\ \infty & \text{else}.
		\end{cases}
	\end{align*}
	Because $z_1,z_2,...,z_{2K}$ are independent, \cite[Exercise 4.2.7]{dembo2009large} yields that for any $K\geq 1$
	\begin{align*}
		p^{(n)} :=  \sqrt{\frac{n\beta}{2b_n}} \left(\frac{1}{\sqrt{2\gamma_nn\beta  }}(z_{1} -2\gamma_n),\frac{1}{2\gamma_n n\beta }z_{1}z_{2} -1,\frac{1}{\sqrt{2\gamma_n n\beta  }} z_{2},...,\frac{1}{\sqrt{2\gamma_nn\beta  }}z_{2K}\right)
	\end{align*}
	satisfies an LDP in $(\RR\times [0,\infty) \times \RR )^{K}$ with speed $b_n$ and good rate function
	\begin{align*}
		\mathcal{I}_3(p_{[3K]}) = \sum_{k=1}^{K} \mathcal{I}_2(p_{3k-2},p_{3k-1},p_{3k})
	\end{align*}
	which is infinite unless $p_{3k}=\xi$ for all $k \geq 1$ and then has the form 
	\begin{align*}
		\mathcal{I}_3(p_{[3K]}) = \frac{1}{2} \sum_{k=1}^{K} p_{3k-2}^2+p_{3k-1}^2.
	\end{align*}
	We now apply the contraction principle with the function
	\begin{align}
		\label{zweiteFktContraction}
		\begin{gathered}
			\Psi: (\RR\times [0,\infty) \times \RR )^{K}\to (\RR \times [0,\infty))^{K-1} \times \RR\\[0,2cm]
			\Psi_1(x) = x_1, \quad \Psi_{2k}(x) = x_{3k-1}, \quad \Psi_{2k+1}(x) = x_{3k}+x_{3k+1}
		\end{gathered}
	\end{align}
	for $k \in \{1,...,K-1\}$, which yields an LDP in $(\RR \times [0,\infty))^{K-1} \times \RR$ for 
	\begin{align*}
		r_{[2K-1]}^{(n)}=\sqrt{\frac{n\beta}{2b_n}}[(\Tilde{d}_1,\Tilde{c}^2_1,\Tilde{d}_2,\Tilde{c}^2_2,...,\Tilde{d}_K)-(0,1,0,1,...,0)]
	\end{align*}
	with speed $b_n$ and good rate function
	\begin{align*}
		\mathcal{I}_4(x) 
		&= \inf\{\mathcal{I}_3(z) : \Psi(z) = x\} 
		= \frac{1}{2}x_{1}^2 +\frac{1}{2} \sum_{k=1}^{2K-1}\left((x_{2k+1}-\xi)^2 + x_{2k}^2 \right)
		= \frac{1}{2} \Vert x_{[2K-1]} - w_{[2K-1]} \Vert_2^2 
	\end{align*}
	where $w = (0,0,\xi,0,\xi,0,\xi,...)^\top$.
	Now, a mapping
	\begin{align}
		\label{AbbMomKoeff}
		\varphi_K:  r_{[2K-1]}=(d_1,c_1^2,...,c^2_{K-1},d_K) \mapsto m_{[2K-1]} 
	\end{align}
	from $(\RR \times [0,\infty))^{K-1} \times \RR$ 
	to the set $\mathbb{M}_{2K-1}$ of first $2K-1$ moments of probability measures can be defined as a composition $\varphi = g \circ h$ where 
	$$g:(\RR \times [0,\infty))^{K-1} \times \RR \to \mathbb{M}_{2K-1}$$ 
	is the map defined by \eqref{eq:spectralmeasuremoments} and $h$ is given by
	\begin{align*}
		h: (\RR \times [0,\infty))^{K-1} \times \RR\to (\RR \times [0,\infty))^{K-1} \times \RR, \quad x \mapsto (x_1,\sqrt{x_2},x_3,\sqrt{x_4},...,x_{2K-1})^\top.
	\end{align*}
	Clearly, the partial derivatives of the $k$-th entry, $1\leq k \leq 2K-1$, of $\varphi_K(r_{[2K-1]})$ exist, since it is a composition of the polynomial $g_k$ and the differentiable function $h_k$.
	The Jacobian matrix of the map $\varphi_K$ evaluated at the Jacobi coefficients of the semicircle law was calculated in \cite{dette2012distributions} 
	as
	\begin{align}
		\label{eq:HadamardDerivative}
		\frac{\partial \varphi_K}{\partial r_{[2K-1]}}(r_{[2K-1]}(\mu_{sc})) = D_{2K-1}
	\end{align}
	where $D_{2K-1}$ is the truncated version of \eqref{eq:D} given by
	\begin{align}
		\label{eq:Dk}
		(D_{2K-1})_{i,j} = 
		\binom{i}{\frac{i-j}{2}} - \binom{i}{\frac{i-j}{2}-1} 
	\end{align}
	if $i \geq j $ and $i+j$ even and $(D_{2K-1})_{i,j}=0$ otherwise. We use the convention $\binom{i}{-1}=0$ and note that $D_{2K-1}$ is nonsingular. Due to the Delta method proven in \cite{gao2011delta}, the sequence of $(2K-1)$-dimensional moment vectors
	\begin{align*}
		\sqrt{\frac{n\beta}{2b_n}}\big(m_{[2K-1]}(\mu_n)-m_{[2K-1]}(\mu_{\text{sc}})\big)
	\end{align*}
	satisfies an LDP with speed $b_n$ and good rate function
	\begin{align*}
		\mathcal{I}_5(m_{[2K-1]}) 
		= \inf\big\{\mathcal{I}_4(x) \mid D_{2K-1}x 
		=m_{[2K-1]}\big\} =\frac{1}{2}\big\Vert D_{2K-1}^{-1} m_{[2K-1]} - w_{[2K-1]} \big\Vert^2_2.
	\end{align*}
\end{proof}

\begin{proof}[Proof of Theorem \ref{thm:MainMDP}]
	Lemma \ref{MomentsMDP} yields the desired MDP for finite moment vectors.
	We begin by showing the integral representation of the rate function, that is
	\begin{align*}
		\frac{1}{2}\left\Vert D^{-1}_{2K-1}\big(m_{[2K-1]}-D_{2K-1}w_{[2K-1]}\big)\right\Vert_2^2
		= \frac{1}{2} \sum_{k=1}^{2K-1} \left( \int q_k(x) d(\sigma_m - \sigma_{Dw})(x)\right)^2.
	\end{align*}
	Here, $\sigma_m$ and $\sigma_{Dw}$ are measures with first $2K-1$ moments given by $m_{[2K-1]}$ and $Dw_{[2K-1]}$, respectively. Existence (but not uniqueness) follows from the solution to the signed moment problem in \cite{boas1939stieltjes}.
	We now repeat the proof given in \cite[p.15]{nagel2013moderatedeviationsspectralmeasures} based on an idea presented in \cite{chang1993normal}. 
	Define $X_i: x \mapsto x^i-m_{i}(\musc)$ for $i \leq 2K-1$ as a random variable on $(\RR,\mathcal{B}(\RR),\musc)$. Then
	\begin{align}
		\label{eq:MDPPaperJan}
		\EE[X_i]=0, \quad \text{Cov}(X_i,X_j) = m_{i+j}(\musc)-m_i(\musc)m_j(\musc) = (D_{2K-1}D_{2K-1}^\top)_{i,j}
	\end{align}
	where the last identity is due to \cite[Lemma 2.3]{nagel2013moderatedeviationsspectralmeasures}. Hence, the entries of $D_{2K-1}^{-1}X$ are uncorrelated with variance 1 which implies that the entries of $D_{2K-1}^{-1}X$ are the first $2K-1$ polynomials $(q_1(x),...,q_{2K-1}(x))$ orthonormal with respect to $\musc$. Thus, the entries in the $i$-th row of $D_{2K-1}$ are given by the coefficients of the $i$-th orthonormal polynomial except for the constant term and for any truncated vector of moments $m_{[2K-1]}(\nu)$ of a signed measure $\nu\in \mathcal M_0^f$ we have
	\begin{align*}
		D_{2K-1}^{-1} &m_{[2K-1]}(\nu) \\
		&= \left(\int q_1(x)-q_1(0)\, d\nu(x), \int q_2(x)-q_2(0)\, d\nu(x),..., \int q_{2K-1}(x)-q_{2K-1}(0)\, d\nu(x)\right)^\top \\
		&= \left(\int q_1(x)\, d\nu(x), \int q_2(x)\, d\nu(x),..., \int q_{2K-1}(x)\, d\nu(x)\right)^\top.
	\end{align*}
	Hence, for any measure $\mu_{m}$ with first $2K-1$ moments given by $m_{[2K-1]}$,
	\begin{align*}
		\mathcal{I}_5(m_{[2K-1]}) &= \frac{1}{2}\left\Vert D^{-1}_{[2K-1]}\big(m_{[2K-1]}-D_{2K-1}w_{[2K-1]}\big)\right\Vert_2^2
		= \frac{1}{2} \sum_{k=1}^{2K-1} \left( \int q_k(x) d(\sigma_m - \sigma_{Dw})(x)\right)^2. 
	\end{align*}
	Finally, the Dawson-Gärtner Theorem, see \cite[Theorem 4.6.1]{dembo2009large}, yields an LDP in $\RR^\NN$ with speed $b_n$ and good rate function
	\begin{align*}
		\mathcal{I}(m)= \sup_{K \in \NN} \mathcal{I}_5(m_{[2K-1]}) = \frac{1}{2} \sum_{k=1}^\infty \left( \int q_k(x) d(\sigma_m - \sigma_{Dw})(x)\right)^2.
	\end{align*}
	It remains to show that 
	\begin{align*}
		(Dw)_k = \begin{cases}
			\xi \, \binom{k}{\frac{k-3}{2}}  & i\text{ odd and } k \geq 3 \\
			\;0 & k \text{ even or } k=1
		\end{cases}
	\end{align*}
	and thus $(Dw)_k$ coincides with $m(\rho_\xi)$ given in \eqref{eq:strangemoments}. Indeed, for $k$ even $(Dw)_k=0$ holds.
	For $k$ odd and $k \geq 3$, we have a telescoping sum
	\begin{align*}
		(Dw)_k = \sum_{l=1}^\infty D_{kl}w_l = \sum_{l=3,\, l \text{ odd}}^k \left(\binom{k}{\frac{k-l}{2}} - \binom{k}{\frac{k-l}{2}-1}\right) \xi
		= \binom{k}{\frac{k-3}{2}} \, \xi.
	\end{align*}
\end{proof}

\subsection{On the alternative rescaling, proof of Theorem \ref{thm:altMainMDP}}
\label{sec:AltScalingProof}
In this section we remark on the necessary modifications to the proofs of the previous section in order to obtain an MDP result for the alternative scaling in Section \ref{sec:AltScaling}.

The diagonal entries of the alternatively rescaled $n$-dimensional Jacobi matrix can be rewritten as 
\begin{align*}
	\hat{d}_k:= \frac{1}{\sqrt{2\gamma_nn\beta  }}(z_{2k-1} -2\gamma_n) + \frac{1}{\sqrt{2\gamma_nn\beta  }}(z_{2k-2}-n\beta)
\end{align*}
for $2 \leq k \leq n$ and 
\begin{align*}
	\hat{d}_1:= \frac{1}{\sqrt{2\gamma_nn\beta  }}(z_{2k-1} -2\gamma_n-n\beta).
\end{align*}
The scaling of the off-diagonal entries is the same as before.
For $k\geq2$, Lemma \ref{MDPLemma} has to be slightly changed such that it may be applied to 
\begin{align*}
	\sqrt{\frac{n\beta }{2b_n}} \left(\frac{1}{\sqrt{2\gamma_nn\beta  }}(z_{2k-1} -2\gamma_n),\frac{1}{2\gamma_n}z_{2k-1}-1,\frac{1}{\sqrt{2\gamma_n n\beta  }} z_{2k}-n\beta,\frac{1}{n\beta } z_{2k} -1\right).
\end{align*}
The only difference to \eqref{eq:GOriginalMDP} is the third entry. A calculation analogous to \eqref{eq:TaylorMDPxi} results in
\begin{align*}
	\hat{\mathcal{C}}(t) = \frac{1}{2} (t_1^2+t_4^2), \quad t \in \RR^4
\end{align*}
instead of the $\mathcal{C}(t)$ given in \eqref{eq:CtMDP}. Thus, we obtain the same good rate function as in Lemma \ref{MDPLemma} except that its infinite unless $x_2 = x_3=0$.
For $k=1$, the argumentation in Lemma \ref{MDPLemma} has to be modified such that it may be applied to 
\begin{align} \label{eq:modfiedz1}
	\sqrt{\frac{n\beta }{2b_n}} \left(\frac{1}{\sqrt{2\gamma_nn\beta  }}(z_{1} -2\gamma_n-n\beta),\frac{1}{2\gamma_n}z_{1}-1,\frac{1}{\sqrt{2\gamma_n n\beta  }} z_{2}-n\beta,\frac{1}{n\beta } z_{2} -1\right)
\end{align}
which leads to 
\begin{align*}
	\hat{\mathcal{C}}(t) = \frac{1}{2} ((t_1-\xi)^2+t_4^2) -\frac{1}{2} \xi^2, \quad t \in \RR^4
\end{align*}
and thus to a good rate function for \eqref{eq:modfiedz1}, which is infinite in $x \in \RR^4$ unless $x_2=x_3 = 0$ and in that case is given by
\begin{align*}
	\mathcal I_0(x) = \frac{1}{2} ((x_1+\xi)^2+x_4^2) .
\end{align*}
We then proceed similarly to the proof of Lemma \ref{MomentsMDP} and obtain an analogous result with $\hat{w}=(-\xi,0,0,0,...)$ replacing $w = (0,0,\xi,0,\xi,0,...)$. The proof of the integral representation of the rate function has the same steps as the proof of Theorem \ref{thm:MainMDP}, but $D \hat{w}$ differs in the odd entries from $Dw = m(\rho_\xi)$. For $k\geq 3$ odd
\begin{align}
	(D\hat{w})_k = \sum_{l=1}^\infty D_{kl}\hat{w}_l = -\xi D_{k,1} = \xi \left( \binom{k}{\frac{k-3}{2}}-\binom{k}{\frac{k-1}{2}}\right).
\end{align}
A helpful reference for finding compactly supported measures with binomial moments is \cite{Penson2014}. In our case, we obtain the moments of $\hat{\rho}_\xi$, given in  \eqref{eq:strangemomentsdensity2}, since
\begin{align*}
	m_k(\hat{\rho}_\xi) = \frac{\xi}{2}\big(m_{k+3}(\muas) - 4m_{k+1}(\muas)\big),
\end{align*}
where $\muas$ denotes the arcsine distribution on $[-2,2]$, see the analogous argumentation for $\rho_\xi$ in \eqref{eq:densityStrangemoments}. A short computation yields $m(\hat{\rho}_{\xi}) = D \hat{w}$.

\subsection{Proof of Theorem \ref{thm:CltPolynoms} (CLT)}
The CLT is proven for the entries of the Jacobi matrix first and subsequently transferred to a finite moment vector. The method is similar to the MDP. In Lemma \ref{MDPLemma}, we first prove an CLT for a finite vector constructed from the entries of $(z_1,z_2,...)$. We can then use the classical Delta method to transfer the CLT to moments. As a final step, we transfer the CLT to polynomial test functions and determine the mean and variance of the limit distribution. 

The set of all moment sequences of $\nu \in \mathcal{M}_0^f$ is given by $\RR^\NN$. Accordingly, the set of truncated moment sequences $m_{[K]}(\mu) = (m_1(\mu),m_2(\mu),...,m_K(\mu))$ is given by $\RR^K$. Both sets may be equipped with the product topology which is metrizable.
\begin{lem}
	\label{weakconvmom}
	Let $(\gamma_n)_n$ be a real sequence that satisfies
	\begin{align*}
		\lim_{n \to \infty}\frac{\gamma_n}{n} = \infty
		\quad \text{and} \quad 
		\lim_{n \to \infty}\frac{n\beta}{2\sqrt{\gamma_n}} = \zeta \in [0,\infty) .
	\end{align*} 
	and let $\mu_n$ denote the spectral measure of the Laguerre ensemble with rescaled eigenvalues as in \eqref{eq:spectralmeasure}.     
	Then the sequence of $(2K-1)$-dimensional random moment vectors $\big(m_{[2K-1]}(\mu_n)\big)_n \subset \RR^{2K-1}$ with $K \in \NN$ satisfies
	\begin{align*}
		\sqrt{n\beta/2}\big( m_{[2K-1]}(\mu_n)-m_{[2K-1]}(\musc)\big) \xrightarrow[n \to \infty]{d} \mathcal{N}(D_{2K-1}w_{[2K-1]},D_{2K-1}D_{2K-1}^\top).
	\end{align*}
	Here $D_{2K-1}$ is the $(2K-1) \times (2K-1)$-dimensional triangular matrix given in \eqref{eq:Dk} and $w_{[2K-1]} = (0,0,\zeta,0,\zeta,0,...\zeta)^\top$.
\end{lem}
\begin{proof}
	Let
	\begin{align*}
		\mathcal{G}^{(n)} = \sqrt{n \beta/2} \left(\frac{1}{\sqrt{2\gamma_n n\beta}}(z_{2k-1} -2\gamma_n),\frac{1}{2\gamma_n}z_{2k-1}-1,\frac{1}{\sqrt{2\gamma_n n\beta}} z_{2k},\frac{1}{n\beta} z_{2k} -1\right).
	\end{align*}
	With the same arguments as in the proof of Lemma \ref{MDPLemma}, it can be proven that the characteristic function of $\mathcal{G}^{(n)}$ converges, such that for $t\in\RR^4$
	\begin{align*}
		\EE \Big[\exp \big(i \langle t, \mathcal{G}^{(n)} \rangle\big) \Big] \xrightarrow[n \to \infty]{} \exp\Big(-\frac{1}{2}(t_1^2+t_4^2)+\zeta it_3\Big).
	\end{align*}
	Thus $\mathcal{G}^{(n)} \xrightarrow[n \to \infty]{d} \mathcal{N}(\mu_1,\Sigma_1) $ with $\mu_1 = (0,0,\zeta,0)^\top$ and
	\begin{align*}
		\begin{matrix}
			\Sigma_1 = \left( \begin{matrix}
				1 & 0 & 0 & 0 \\
				0 & 0 & 0 & 0 \\
				0 & 0 & 0 & 0 \\
				0 & 0 & 0 & 1 
			\end{matrix}\right).
		\end{matrix}
	\end{align*}
	Next, we apply the continous mapping theorem and the classical Delta method with the same functions which were used in the proof of Theorem \ref{thm:MainMDP} for the contraction principle and the Delta method for MDPs, respectively. First, the classical Delta method with the function in (\ref{ErsteFktDelta}) yields
	\begin{align*}
		\sqrt{n\beta/2} \left( \Big(\frac{1}{\sqrt{2\gamma_n n\beta}}(z_{2k-1} -2\gamma_n),\frac{1}{2\gamma_n n \beta}z_{2k-1}z_{2k},\frac{1}{\sqrt{2\gamma_n n\beta}} z_{2k} \Big) -(0,1,0)\right) \xrightarrow[n \to \infty]{d} \mathcal{N}(0,\Sigma_2)
	\end{align*}
	with $\mu_2 = (0,0,\zeta)^\top$ and
	\begin{align*}
		\begin{matrix}
			\Sigma_2 = \left( \begin{matrix}
				1 & 0 & 0  \\
				0 & 1 & 0 \\
				0 & 0 & 0 
			\end{matrix}\right).
		\end{matrix}
	\end{align*}
	Because $z_1,z_2,...,z_{2K}$ are independent, we obtain
	\begin{align*}
		\sqrt{n\beta/2}  \left( \Big(
		\frac{z_{1} -2\gamma_n}{\sqrt{2\gamma_n n\beta}},
		\frac{z_{1}z_{2}}{2\gamma_n n \beta},
		\frac{z_{2}}{\sqrt{2\gamma_n n\beta}},
		...,
		\frac{z_{2K} }{\sqrt{2\gamma_n n\beta}}\Big) 
		-(0,1,0,...,0,1,0)\right)
		\xrightarrow[n \to \infty]{d} \mathcal{N}(\mu_3,\Sigma_3)
	\end{align*}
	where $\Sigma_3$ is the $3K\times3K$ diagonal matrix with diagonal entries $(1,1,0,...,1,1,0)$ and
	\begin{align*}
		\mu_3 = (0,0,\zeta,...,0,0,\zeta)^\top.
	\end{align*}
	Then, the continous mapping theorem with the function in (\ref{zweiteFktContraction}) yields
	\begin{align*}
		\sqrt{n\beta/2}  &\left( \Big(
		\frac{z_{1} -2\gamma_n}{\sqrt{2\gamma_n n\beta}},
		\frac{z_{1}z_{2}}{2\gamma_n n \beta},
		\frac{z_{2}+z_3-2\gamma_n}{\sqrt{2\gamma_n n\beta}},
		...,
		\frac{z_{2K-2}+z_{2K-1} -2\gamma_n}{\sqrt{2\gamma_n n\beta}}
		\Big) 
		-(0,1,0,1,...,0)\right) \\[2mm]
		&\xrightarrow[n \to \infty]{d} \mathcal{N}(\mu_4,I_{2K-1})
	\end{align*}
	where $I_{2K-1}$ is the $(2K-1)\times(2K-1)$ identity matrix and
	\begin{align*}
		\mu_4 = w_{[2K-1]} = (0,0,\zeta,0,\zeta,0,\zeta,...,\zeta)^\top.
	\end{align*}
	Finally, the classical Delta method with the function in (\ref{AbbMomKoeff}) and the derivative given in (\ref{eq:HadamardDerivative}) yield the desired result.
\end{proof}
\begin{proof}[Proof of Theorem \ref{thm:CltPolynoms}]
	Let $p$ be a polynomial of the form $p(x) = \sum_{i=0}^d \alpha_i x^i$ with $\alpha_i \in \RR$. Then for $K \in \NN$ with $d \leq 2K-1$
	\begin{align*}
		\int p \,d\mu = \sum_{i=1}^d \alpha_i m_{i}(\mu) = \sum_{i=1}^{2K-1} \alpha_i m_{i}(\mu)  \quad \text{for } \mu \in \mathcal{M}_1^c,
	\end{align*}
	where $\alpha_i = 0$ for $d<i\leq 2K-1$.
	Thus, we obtain from Lemma \ref{weakconvmom} the convergence
	\begin{align*}
		\sqrt{n\beta/2}\left(\int p \, d\mu_n - \int p \, d\musc \right) \xrightarrow[n \to \infty]{d}
		\mathcal{N}(\alpha^\top D_{2K-1}w_{[2K-1]},\alpha^\top D_{2K-1}D_{2K-1}^\top \alpha)
	\end{align*}
	with $\alpha = (\alpha_1,...,\alpha_{2K-1})^\top$. Using the idea from \cite{chang1993normal}, which already appeared in the proof of Theorem \ref{thm:MainMDP}, define $X_i: x \mapsto x^i-m_{i}(\musc)$ for $i \in \NN$ as a random variable on $(\RR,\mathcal{B}(\RR),\musc)$. Then, as in \eqref{eq:MDPPaperJan},
	\begin{align*}
		\EE[X_i]=0, \quad \text{Cov}(X_i,X_j) = m_{i+j}(\musc)-m_i(\musc)m_j(\musc) = (D_{2K-1}D_{2K-1}^\top)_{i,j}.
	\end{align*}
	Therefore, the variance has the form
	\begin{align*}
		\alpha^\top\big(D_{2K-1}D_{2K-1}^\top\big)\alpha &= \text{Cov}\big(\alpha^\top X,\alpha^\top X\big) = \text{Var}\left(\sum_{i=1}^{2K-1}\alpha_i\big(x^i-m_{i}(\musc)\big)\right) \\
		&= \int\left(p-\int p \,d\musc\right)^2 \,d\musc.
	\end{align*}
	The form of the expected value follows from \eqref{eq:strangemoments}.
\end{proof}

\medskip

\textbf{Acknowledgments:} 
We thank Johannes Heiny for pointing us towards references \cite{enriquez2016spectra,noiry2018spectral}.  
This work was funded by the Deutsche Forschungsgemeinschaft (DFG, German Research Foundation) - Projektnummer 499508288

\bibliographystyle{alpha}
\bibliography{Literatur}

\bigskip

{\footnotesize
	\noindent
	TU Dortmund, \\
	Fakult\"at f\"ur Mathematik, \\
	Vogelpothsweg 87, 
	44227 Dortmund, 
	Germany, \\
	{\tt helene.goetz@tu-dortmund.de}\\
	{\tt jan.nagel@tu-dortmund.de }
}

\end{document}